\documentclass{amsart}
\usepackage{color}
\usepackage{amscd}
\usepackage{graphicx}
\usepackage{hyperref}
\usepackage{amssymb}
\usepackage{amsfonts}
\usepackage{euscript}
\usepackage[all]{xy}
\usepackage[cal=boondoxo]{mathalfa}

\numberwithin{equation}{section}
\renewcommand{\subsection}[1]{\hspace{-\parindent}\refstepcounter{subsection}{\bf (\arabic{section}\alph{subsection}) #1.}\addcontentsline{toc}{subsection}{\bf #1.}}

\newenvironment{nouppercase}{%
  \renewcommand{\uppercasenonmath}[1]{}}{}

\theoremstyle{plain}
\newtheorem{thm}{Theorem}[section]
\newtheorem{theorem}[thm]{Theorem}
\newtheorem{application}[thm]{Application}

\newtheorem{corollary}[thm]{Corollary}

\newtheorem{example}[thm]{Example}

\newtheorem{lemma}[thm]{Lemma}

\newtheorem*{claim*}{Claim} 
\newtheorem*{lemma*}{Lemma}
\newtheorem*{theorem*}{Theorem}
\newtheorem*{conjecture*}{Conjecture}
\newtheorem{remark}[thm]{Remark}

\newcommand{\bC}{{\mathbb C}}

\newcommand{\bQ}{{\mathbb Q}}
\newcommand{\bR}{{\mathbb R}}

\newcommand{\bZ}{{\mathbb Z}}

\newcommand{\scrH}{\mathcal H}

\newcommand{\scrJ}{\mathcal J}

\newcommand{\scrP}{\mathcal P}

\newcommand{\scrR}{\mathcal R}

\newcommand{\scrX}{\mathcal X}

\newcommand{\scrZ}{\mathcal Z}

\newcommand{\scrl}{\mathcal l}
\newcommand{\scrs}{\mathcal s}

\newcommand{\half}{{\textstyle\frac{1}{2}}}

\newcommand{\iso}{\cong}
\newcommand{\htp}{\simeq}
\newcommand{\smooth}{C^\infty}

\title[CONNECTIONS]{\Large\larger\rm Connections on equivariant \\ Hamiltonian Floer cohomology}
\author{Paul Seidel}

\begin{document}
\begin{nouppercase}
\maketitle
\end{nouppercase}

\begin{abstract}
We construct connections on $S^1$-equivariant Hamiltonian Floer cohomology, which differentiate with respect to certain formal parameters.
\end{abstract}

\section{Introduction\label{sec:intro}}

Floer cohomology often involves formal parameters, which take into account various topological features. This paper concerns differentiation with respect to such parameters. Before we turn to that, it may be appropriate to recall other contexts in which cohomology groups come with similar differentiation operations:
\begin{itemize}
\itemsep.5em
\item
In algebraic geometry, given a smooth family of algebraic varieties, the fibrewise algebraic de Rham cohomology carries the Gauss-Manin connection \cite{katz-oda68}. 
Griffiths transversality \cite{griffiths68, griffiths70} measures the failure of the Hodge filtration to be covariantly constant, and that is the starting point for the theory of variations of Hodge structures. 

\item
The Gauss-Manin connection has been generalized to noncommutative geometry by Getzler \cite{getzler95}, where it lives on the periodic cyclic homology of a family of dg (or $A_\infty$) algebras. Recall that periodic cyclic homology can be obtained from negative cyclic homology by inverting a formal parameter, here denoted by $u$. In Getzler's formula, only a simple pole $u^{-1}$ appears (hence, $u$ times that connection is an operation on negative cyclic homology). This property is the analogue of Griffiths transversality, in the formalism of variations of semi-infinite Hodge structures \cite{barannikov01}. 

\item
There is a related but distinct connection on periodic cyclic homology, which applies to a single dg algebra, and differentiates in $u$-direction \cite{katzarkov-kontsevich-pantev08, shklyarov14}. More precisely, the connection is defined for $\bZ/2$-graded dg algebras (and is basically trivial if the grading can be lifted to $\bZ$). It involves a $u^{-2}$ term, hence can be thought of as having (in general) an irregular singularity at the parameter value $u = 0$. In algebraic geometry, a related construction appears in the context of exponentially twisted de Rham cohomology.

\item
Closer to our interests is the (small) quantum connection in Gromov-Witten theory \cite{givental94, dubrovin98}. This differentiates in direction of the Novikov parameters, as well as another parameter, which one can think of as being our previous $u$. In the Calabi-Yau case, where differentiation in $u$-direction is not interesting, \cite{ganatra-perutz-sheridan15} announced a proof of the fact that the quantum connection is related to Getzler's connection on the cyclic homology of the Fukaya category, through the (cyclic) open-closed string map.
\end{itemize}

The aim of this paper is to construct connections on $S^1$-equivariant Hamiltonian Floer cohomology. The idea underlying the construction is quite general, since it mainly involves certain chain level TQFT operations (geometric realizations of the Cartan calculus in noncommutative geometry, 
which underpins Getzler's construction). However, we will not aim for maximal generality; instead, we illustrate the idea by two specific instances, leading, in slightly different contexts, to what we call the $q$-connection and $u$-connection. In cases where Floer cohomology reduces to ordinary cohomology, these reproduce appropriately specialized versions of the quantum connection. (One also expects them to be related to the corresponding structures in noncommutative geometry through open-closed string maps, but we will not pursue that direction in this paper.)
%

The original motivation comes from \cite{seidel16}. That paper considers (non-equivariant) symplectic cohomology, which is a specific instance or application of Hamiltonian Floer cohomology. One imposes a crucial additional assumption, which is that the class of the symplectic form should map to zero in symplectic cohomology. One then gets a connection on that cohomology, which is not canonical (it depends on the choice of an appropriate bounding cochain, which ``certifies'' the previously mentioned vanishing assumption). This looks somewhat different from our $q$-connection, which only exists for the $S^1$-equivariant theory, does not require any additional assumption, and is canonical. In spite of that,  one still expects to be able to relate the two connections, by means of a suitable intermediate object; see \cite[Section 3]{seidel16}.
The analogous situation in algebraic geometry would be the case of a family of smooth varieties with vanishing  Kodaira-Spencer class (this means that the family is infinitesimally trivial, but not necessarily globally trivial; after all, any family of affine varieties satisfies that condition). In that case, one can define a non-canonical connection on the spaces of fibrewise (algebraic) differential forms, which (in a suitable sense) induces the Gauss-Manin connection.

The structure of this paper is as follows. Section \ref{sec:results} introduces the relevant geometric situation, and states our main results. Section \ref{sec:morse} collects some background material about Morse theory on $\bC P^\infty$. Section \ref{sec:floer} is a review of some relevant aspects of Floer theory. Section \ref{sec:q} defines the $q$-connection, and in Section \ref{sec:u}, we adapt the previous argument to get the $u$-connection.

{\em Acknowledgments.} This work was partially supported by the Simons Foundation, through a Simons Investigator award; by NSF grant DMS-1500954; and by the Institute for Advanced Study, through a visiting appointment supported by grants from the Ambrose Monell Foundation and the Simonyi Endowment Fund. I would like to thank Sheel Ganatra and Nick Sheridan for illuminating discussions.

\section{Main constructions\label{sec:results}}
After recalling some basic Floer-theoretic notions and terminology, we explain the formal structure of the operations to be constructed. We also include a few comments about the wider context into which they fit (implications; relations with other developments; and possible generalizations).

\subsection{The $q$-connection\label{subsec:the-q}}
Let $(M,\omega)$ be a compact symplectic manifold with convex contact type boundary. (The boundary could be empty, even though that case is of less interest for us; also, we will not really make any use of the contact geometry of the boundary, other than to ensure suitable convexity properties for solutions of Cauchy-Riemann equations.) For technical simplicity, and also to strengthen the similarity with classical cohomology, we will assume that 
\begin{equation} \label{eq:cy}
c_1(M) = 0.
\end{equation}

Let $A \subset \bR$ be the additive subgroup generated by the integers and the periods $\omega \cdot H_2(M;\bZ)$, and $R \subset \bR$ the subring with the same generators. Clearly, $[\omega] \in H^2(M;\bR)$ can be lifted to $H^2(M;A)$; we pick such a lift, denoted by $[\Omega]$. We use a single-variable Novikov ring $\Lambda$ where the coefficients lie in $R$, and the exponents in $A$. This means that elements of $\Lambda$ are formal series
\begin{equation} \label{eq:novikov}
f(q) = r_0 q^{a_0} + r_1 q^{a_1} + \cdots, \quad r_i \in R, \;\; a_i \in A, \;\; \textstyle \lim_i a_i = +\infty.
\end{equation}
By construction, $\Lambda$ is closed under differentiation $\partial_q$.
Quantum cohomology is the graded ring obtained by equipping $H^*(M;\Lambda)$ with the (small) quantum product $\ast$. Take $u$ to be a formal variable of degree $2$, and extend the quantum product $u$-linearly to $H^*(M;\Lambda[[u]])$. Concerning the notation, let's point out that the distinction between polynomials and power series in $u$ is strictly speaking irrelevant here, because of the grading: in each degree, only finitely many powers of $u$ can appear. In spite of that, we keep the power series notation since it's appropriate in a more general context; the same will apply to Floer cohomology. The quantum connection (or rather, the part of that connection which concerns us at this point) is the endomorphism
\begin{equation} \label{eq:quantum-1}
\begin{aligned}
& D_q: H^*(M;\Lambda[[u]]) \longrightarrow H^{*+2}(M;\Lambda[[u]]), \\
& D_q x = u\partial_q x + q^{-1}[\Omega] \ast x.
\end{aligned}
\end{equation}
As defined, $D_q$ is a connection in $u\partial_q$-direction, which may feel awkward. If one wants to get a connection in the more standard sense (which means in $\partial_q$-direction, hence having degree $0$), one can instead take $u^{-1}D_q$, acting on $H^*(M;\Lambda((u)))$.

We will consider Floer cohomology groups $\mathit{HF}^*(M,\epsilon)$, for $\epsilon>0$, which are defined using functions whose Hamiltonian vector field restricts to $\epsilon$ times the Reeb field on $\partial M$ (assuming that there are no Reeb chords of length $\epsilon$). Each such group is a finitely generated $\bZ$-graded $\Lambda$-module. It also carries the structure of a module over the quantum cohomology ring, via the quantum cap product, which we will write as $\frown$. Our main object of study is the $S^1$-equivariant version of Floer cohomology, denoted  by $\mathit{HF}^*_{\mathit{eq}}(M, \epsilon)$. This is a finitely generated $\bZ$-graded module over $\Lambda[[u]]$. Equivariant Floer cohomology sits in a long exact sequence
\begin{equation} \label{eq:u-sequence}
\cdots \rightarrow \mathit{HF}^{*-2}_{\mathit{eq}}(M, \epsilon) \stackrel{u}{\longrightarrow} \mathit{HF}^*_{\mathit{eq}}(M, \epsilon) \longrightarrow \mathit{HF}^*(M, \epsilon) \longrightarrow \mathit{HF}^{*-1}_{\mathit{eq}}(M, \epsilon) \rightarrow \cdots
\end{equation}
We will often make use of the forgetful map (from equivariant to ordinary Floer cohomology) which is part of that sequence. Also of interest are the PSS maps, which are canonical maps
\begin{equation}
\label{eq:morse-floer} 
\xymatrix{
H^*(M;\Lambda[[u]]) \ar[d]^-{u = 0} \ar[rr]^-{B_{\mathit{eq}}} && \mathit{HF}^*_{\mathit{eq}}(M, \epsilon) \ar[d] \\
H^*(M;\Lambda) \ar[rr]^-{B} && \mathit{HF}^*(M, \epsilon).
}
\end{equation}
$B$ relates the quantum product with its cap product counterpart. Moreover, if we choose $\epsilon$ small, then both $B$ and $B_{\mathit{eq}}$ are isomorphisms. The $q$-connection can be described as follows:

\begin{theorem} \label{th:q}
There is a canonical additive endomorphism
\begin{equation} \label{eq:q-connection}
\begin{aligned}
& \Gamma_q: \mathit{HF}^*_{\mathit{eq}}(M, \epsilon) \longrightarrow \mathit{HF}^{*+2}_{\mathit{eq}}(M, \epsilon), \\
& \Gamma_q( f x) = f \Gamma_q(x) + u (\partial_q f) x \quad \text{for $f \in \Lambda[[u]]$,}
\end{aligned}
\end{equation}
which fits into a commutative diagram
\begin{equation} \label{eq:q-zero}
\xymatrix{ 
\ar[d] \mathit{HF}^*_{\mathit{eq}}(M, \epsilon) \ar[rr]^-{\Gamma_q} && \mathit{HF}^{*+2}_{\mathit{eq}}(M, \epsilon) \ar[d] \\
\mathit{HF}^*(M, \epsilon) \ar[rr]^-{q^{-1}[\Omega] \frown \cdot} && \mathit{HF}^{*+2}(M, \epsilon).
}
\end{equation}
Moreover, for small $\epsilon$, the isomorphism $B_{\mathit{eq}}$ identifies $\Gamma_q$ with $D_q$.
\end{theorem}

Let's give at least a hint of the construction. Floer cochain complexes are, by definition, complexes of free modules over the Novikov ring, with a distinguished basis (up to signs). Using that basis, one can equip them with the naive operation of differentiation $\partial_q$ in the Novikov variable, but that operation does not commute with the Floer differential $d$. In our version of the definition, $d$ counts Floer trajectories with weights $\pm q^E$, where $E$ is the intersection number with a suitable cycle $\Omega$ representing the symplectic class. Clearly, the commutator $\partial_q d - d \partial_q$ counts those same trajectories with weights $\partial_q (\pm q^E) = \pm (q^{-1} E) q^E$. The idea is to interpret this new count as a kind of ``Lie action'' of the cohomology class $q^{-1}[\Omega]$. On the $S^1$-equivariant complex, the Lie action operation becomes nullhomotopic after multiplication with the equivariant formal parameter $u$. One uses the nullhomotopy to add a correction term to $u\partial_q$, turning it into a chain map, which induces $\Gamma_q$.

As one can see from this sketch, the $q$-connection is closely tied to the origin of Novikov rings as a way of keeping track of energy, hence to the non-exactness of the symplectic form. If $\omega$ is exact, one can take $\Omega = \emptyset$, in which case the coefficients of the Floer differential are $\pm 1$; then $\partial_q$ is already a chain map, and on cohomology, one has
\begin{equation} \label{eq:exact}
\Gamma_q = u\partial_q.
\end{equation}
Similarly, suppose that the periods are $\omega \cdot H_2(M;\bZ) = m\bZ$, for some integer $m \geq 2$ (and accordingly choose $[\Omega]$ to be the $m$-fold multiple of an integral class). In that case, $\Lambda = \bZ((q))$, but Floer cohomology can in fact be defined using only powers of $q^m$. As a consequence, if we consider the version of the theory with coefficients mod $m$, which we denote by $\mathit{HF}^*(M,\epsilon; \bZ/m)$ (even though it is actually defined using $\Lambda \otimes_{\bZ} \bZ/m = (\bZ/m\bZ)((q))$ as coefficient ring), then that version again carries a trivial $\partial_q$-operation. The same holds for the equivariant theory, and we have a commutative diagram which describes ``$\Gamma_q$ modulo $m$'':
\begin{equation} \label{eq:exact-mod-m}
\xymatrix{
\ar[d] \mathit{HF}^*_{\mathit{eq}}(M, \epsilon) \ar[rr]^-{\Gamma_q} && \mathit{HF}^{*+2}_{\mathit{eq}}(M, \epsilon) \ar[d] \\
\mathit{HF}^*_{\mathit{eq}}(M,\epsilon;\bZ/m) \ar[rr]^-{u\partial_q} && \mathit{HF}^{*+2}_{\mathit{eq}}(M,\epsilon; \bZ/m).
}
\end{equation}
 
If one wants a connection in $\partial_q$-direction, one can consider $u^{-1}\Gamma_q$, acting on 
\begin{equation} \label{eq:localise}
\mathit{HF}^*_{\mathit{eq}}(M, \epsilon) \otimes_{\bZ[[u]]} \bZ((u)) =
\mathit{HF}^*_{\mathit{eq}}(M, \epsilon) \otimes_{\Lambda[[u]]} \Lambda((u)). 
\end{equation}
In this context, we should mention how this fits in with the localisation theorem of \cite{zhao14, albers-cieliebak-frauenfelder16} (even though that will not be pursued further in the body of the paper). An appropriate generalization of that theorem shows that, after tensoring with $\bQ((u))$ instead of $\bZ((u))$ in \eqref{eq:localise}, the equivariant PSS map \eqref{eq:morse-floer} becomes an isomorphism for all $\epsilon$. Moreover, a generalization of the compatibility statement from Theorem \ref{th:q} (not proved here, but not tremendously hard) shows that this map still relates $D_q$ and $\Gamma_q$. Hence, the resulting version of $\Gamma_q$ can be recovered, up to isomorphism, from the standard Gromov-Witten theory of $M$. The ``up to isomorphism'' issue is not negligible, since it may not be easy to see what the localisation isomorphism does to geometrically relevant symplectic cohomology classes (see \cite[Section 3]{seidel16} for an example of this, involving Borman-Sheridan classes). Leaving that aside, note that tensoring with $\bQ((u))$ entails some loss of information ($\bZ$-torsion and $u$-torsion); it seems unlikely that $\Gamma_q$ itself has a description in terms of the Gromov-Witten theory of $M$.

We want to briefly mention some potential further developments. One could extend the construction to multivariable Novikov rings; this corresponds to the version of \eqref{eq:quantum-1} which uses the quantum product with all of $H^2(M;\Lambda)$. A genuinely new question that arises in the multivariable context is that of the (expected) flatness of the connection. It is also worth noting that the construction applies outside the context of Novikov completions as well. For instance, consider the case of an exact symplectic manifold. One can then define Floer cohomology with coefficients in the Laurent polynomial ring over $H_2(M;\bZ)/\mathit{torsion}$. For the $S^1$-equivariant version of that Floer cohomology theory, there is a connection which differentiates in all $H^2(M;\bZ)/\mathit{torsion}$ directions. An analogous idea may apply to string topology (where one studies the $S^1$-equivariant homology of a free loop space, with twisted coefficients).

\subsection{The $u$-connection\label{subsec:u}}
Let's replace \eqref{eq:cy} by the assumption that our symplectic manifold should be either exact or monotone, meaning that
\begin{equation} \label{eq:monotone}
[\omega] = \gamma\, c_1(M) \in H^2(M;\bR), \quad \text{for some $\gamma \geq 0$.}
\end{equation}
In this case, the quantum product can be defined without using the Novikov parameter, as a $\bZ/2$-graded product on $H^*(M;\bZ)$. We will use a different form of the quantum connection this time, namely the endomorphism of the $\bZ/2$-graded group $H^*(M;\bZ[[u]])$ given by
\begin{equation} \label{eq:quantum-2}
D_u x = 2u^2 \partial_u x - 2c_1(M) \ast x + u\mu(x),
\end{equation}
where
\begin{equation} \label{eq:mu}
\mu(x) = k x \quad \text{if $x \in H^k(M;\bZ) \otimes \bZ[[u]]$.}
\end{equation}

The assumption \eqref{eq:monotone} also allows us to define Floer cohomology and its equivariant cousin without using Novikov coefficients, as a finitely generated $\bZ/2$-graded abelian group and finitely generated $\bZ/2$-graded $\bZ[[u]]$-module, respectively (in spite of that difference in the formal setup, we will keep the same notation for them as before). There is also a $\bZ/2$-graded analogue of \eqref{eq:morse-floer}, involving $\bZ$ and $\bZ[[u]]$ as coefficient rings. The counterpart of Theorem \ref{th:q}, describing the basic properties of the $u$-connection, is:

\begin{theorem} \label{th:u}
There is a canonical $\bZ/2$-graded additive endomorphism
\begin{equation} \label{eq:u-connection}
\begin{aligned}
& \Gamma_u: \mathit{HF}^*_{\mathit{eq}}(M, \epsilon) \longrightarrow \mathit{HF}^*_{\mathit{eq}}(M, \epsilon), \\
& \Gamma_u( f x) = f \Gamma_u(x) + 2u^2 (\partial_u f) x \quad \text{for $f \in \bZ[[u]]$,}
\end{aligned}
\end{equation}
which fits into a commutative diagram
\begin{equation} \label{eq:u-zero}
\xymatrix{ 
\ar[d] \mathit{HF}^*_{\mathit{eq}}(M, \epsilon) \ar[rr]^-{\Gamma_u} && \mathit{HF}^*_{\mathit{eq}}(M, \epsilon) \ar[d] \\
\mathit{HF}^*(M, \epsilon) \ar[rr]^-{-2c_1(M) \frown \cdot} && \mathit{HF}^*(M, \epsilon).
}
\end{equation}
For small $\epsilon$, the isomorphism $B_{\mathit{eq}}$ identifies $\Gamma_u$ with $D_u$.
\end{theorem}

The $u$-connection is closely tied to the issue of gradings on Floer cohomology. If $c_1(M) = 0$ (which in our context implies that $[\omega]$ must vanish as well), one has $\bZ$-gradings as in Section \ref{subsec:the-q}. Let $\mathrm{deg}$ be the associated grading operator, which multiplies each element of $\mathit{HF}^*_{\mathit{eq}}(M, \epsilon)$ by its degree. Then, there is a disappointingly simple formula
\begin{equation} \label{eq:graded-gamma}
\Gamma_u(x) = u \,\mathrm{deg}(x).
\end{equation}
More generally, suppose that $c_1(M)$ is $m$ times some class in $H^2(M;\bZ)$, where $m \geq 1$ (of course, the $m = 1$ case always applies). A choice of such a class yields a $(\bZ/2m)$-grading, and one has a diagram analogous to \eqref{eq:exact-mod-m}:
\begin{equation} \label{eq:c1-mod-2m}
\xymatrix{
\ar[d] \mathit{HF}^*_{\mathit{eq}}(M, \epsilon) \ar[rr]^-{\Gamma_u} && \mathit{HF}^{*+2}_{\mathit{eq}}(M, \epsilon) \ar[d] \\
\mathit{HF}^*_{\mathit{eq}}(M,\epsilon;\bZ/2m) \ar[rr]^-{u\,\mathrm{deg}} && \mathit{HF}^{*+2}_{\mathit{eq}}(M,\epsilon; \bZ/2m).
}
\end{equation}
%
%
%
%
%
%
%

Let's assume that our symplectic manifold is monotone, which means \eqref{eq:monotone} with $\gamma>0$. In fact, let's normalize the symplectic form so that
\begin{equation} \label{eq:monotone-2}
[\omega] = c_1(M).
\end{equation}
One can then define a version of quantum cohomology which is $\bZ$-graded but periodic, by adding a formal variable $q$ of degree $2$. More precisely, we want to think of this as a ring structure on the graded $u$-adic completion of $H^*(M;\bZ[q,q^{-1},u])$, which we write as $H^*(M;\bZ[q,q^{-1}][[q^{-1}u]])$. This carries a (degree $2$) operation $D_q$ as in \eqref{eq:quantum-1}. Let $\mathrm{deg}_q$ be the grading operator on $H^*(M;\bZ[q,q^{-1}][[q^{-1}u]])$. Unlike \eqref{eq:mu} this takes the gradings $|q| = |u| = 2$ into account, so one can write it as
\begin{equation} \label{eq:deg-mu}
\mathrm{deg}_q = \mu + 2 u\partial_u + 2 q\partial_q.
\end{equation}
Using \eqref{eq:monotone-2}, one then has
\begin{equation} \label{eq:uq}
D_u = \big( u\, \mathrm{deg}_q - 2q D_q \big)_{q = 1}.
\end{equation}
What this means is: the expression in brackets is $\bZ[q,q^{-1}]$-linear, hence can be specialized to $q = 1$ (which simply means reducing the grading back to $\bZ/2$), and the result then agrees with the previously defined $D_u$. One can similarly define a version of Floer cohomology which is a $\bZ$-graded module over $\bZ[q,q^{-1}]$; and of equivariant Floer cohomology, over $\bZ[q,q^{-1}][[q^{-1}u]]$. The equivariant version carries a $q$-connection as in \eqref{eq:q-connection}. In parallel with \eqref{eq:uq}, this turns out to be related to the $u$-connection, 
\begin{equation} \label{eq:ueq2}
\Gamma_u = \big( u\, \mathrm{deg}_q - 2 q \Gamma_q \big)_{q = 1}.
\end{equation}

We want to make one more observation concerning the monotone case \eqref{eq:monotone-2}. In our original framework \eqref{eq:cy}, Floer cohomology was $\bZ$-graded, and gradings forced all $u$-series to be finite. A similar, but slightly more subtle, principle is at work in the monotone situation, allowing us (after making appropriately careful choices) to define a polynomial version of equivariant Floer cohomology, denoted by $\mathit{HF}^*_{\mathit{poly}}(M, \epsilon)$, which is a finitely generated $\bZ/2$-graded module over $\bZ[u]$, and from which the previous version is recovered by completion:
\begin{equation}
\mathit{HF}^*_{\mathit{eq}}(M, \epsilon) \iso \mathit{HF}^*_{\mathit{poly}}(M, \epsilon) \otimes_{\bZ[u]} \bZ[[u]].
\end{equation}
Similarly, one can define a $u$-connection on $\mathit{HF}^*_{\mathit{poly}}(M, \epsilon)$, of which our previously considered $\Gamma_u$ is the formal germ at $u = 0$. This is of interest because the polynomial (or indeed complex-analytic) theory of irregular connections is much richer than the formal theory (for applications of this theory to $D_u$, see e.g.\ \cite{galkin-golyshev-iritani16}). More immediately, the existence of the polynomial version of the $u$-connection has the following consequence:

\begin{corollary} \label{th:nonzero-u}
As a $\bZ[u]$-module, $\mathit{HF}^*_{\mathit{poly}}(M, \epsilon)$ cannot contain any direct summands  isomorphic to one of the following:
\begin{equation}
\begin{aligned}
& \bZ[u]/(u-\lambda)^d && \text{for $\lambda \neq 0$, and $d \geq 1$; or} \\
& (\bZ/p)[u]/(u-\lambda)^d && \text{for an odd prime $p$, and $d, \lambda$ both coprime to $p$.}
\end{aligned}
\end{equation}
\end{corollary}

This may be a bit of a letdown, since such summands would yield extra information specific to the polynomial theory. However, for the $\bZ$-torsion part, not all such extra information is ruled out by Corollary \ref{th:nonzero-u} (and the remaining possibilites are known to occur in other contexts; see the example of $\bZ/2$-equivariant Lagrangian Floer cohomology in \cite[Section 7c]{seidel14c}). The proof is a one-liner: if $x$ were the generator of such a summand, then
\begin{equation}
0 = \Gamma_u ( (u-\lambda)^d x) = (u-\lambda)^d \Gamma_u x +  2d u^2 (u-\lambda)^{d-1} x.
\end{equation}
If one projects back to the relevant summand, the first term on the right hand side vanishes, while the second does not. For a more geometric view, let's replace $\bZ$ by $\bC$. Then, the idea is that, since the vector field $2u^2 \partial_u$ only vanishes at $u = 0$, a coherent sheaf that admits a connection in the direction of that vector field can't have torsion anywhere else (as in our discussion of $D_q$, it would be interesting to see how this relates to what one might get from localisation techniques).

\section{Morse-theoretic moduli spaces\label{sec:morse}}
Following a familiar strategy (compare e.g.\ \cite{bourgeois-oancea12, seidel-smith10, seidel14c}; in the last two references, the group involved is $\bZ/2$ rather than $S^1$), much of our discussion of $S^1$-equivariant Floer cohomology will be based on the Morse theory of $BS^1 = \bC P^\infty$. In this section, we use this Morse theory to produce various hierarchies of manifolds with corners (of course, one could also try to construct those manifolds directly in a combinatorial way, but that approach seems less natural).

\subsection{Setup}
The basic notation is:
\begin{align}
& 
\bC^\infty = \{w = (w_0,w_1,\dots)\; : \; \text{$w_j \in \bC$ vanishes for almost all $j$}\}, \\
&
B^\infty = \{w \in \bC^\infty\;:\; \|w\|^2 = \textstyle \sum_j |w_j|^2 \leq 1\}, \\
&
S^\infty = \partial B^\infty, \\
&
\bC P^\infty = S^\infty/S^1.
\end{align}
An important ingredient for us will be the shift self-embedding 
\begin{equation} \label{eq:shift-embedding}
\sigma(w_0,w_1,\dots) = (0,w_0,w_1,\dots)
\end{equation}
(we will allow a slight ambiguity in the notation here, using $\sigma$ for the shift acting on either of the spaces above). The quotient map will be denoted by
\begin{equation} \label{eq:q-bundle}
q: S^\infty \longrightarrow \bC P^\infty.
\end{equation}
Let $c_k \in \bC P^\infty$ be the $k$-th unit vector ($k \geq 0$). We will identify the fibre of \eqref{eq:q-bundle} over $c_k$ with $S^1$ in the obvious way. A notational remark is appropriate at his point. Following Floer theory conventions, we set $S^1 = \bR/\bZ$ throughout, so the identification is written as
\begin{equation} \label{eq:ck-identification}
\begin{aligned}
& S^1 \stackrel{\iso}{\longrightarrow} q^{-1}(c_k), \\
& r \longmapsto (0,\dots,e^{2\pi i r},0,\dots).
\end{aligned}
\end{equation}

We will use a specific complex hyperplane in $\bC P^\infty$, as well as a real hypersurface bounding its preimage in $S^\infty$. These are given by, respectively,
\begin{align}
\label{eq:h-hypersurface}
& \textstyle
H = \{\sum_j w_j = 0\} \subset \bC P^\infty, \\
& \textstyle
S = q^{-1}(H) \subset S^\infty, \\
\label{eq:b-bounding}
& \textstyle
B = \{\sum_j w_j \leq 0\} \subset S^\infty.
\end{align}
Clearly, $H \iso \bC P^\infty$ and $S \iso S^\infty$. One also has $B \iso B^\infty$, for instance by a suitable stereographic projection (away from $(1,0,\dots,0)$, to the linear subspace where the sum of all coordinates is zero):
\begin{equation} \label{eq:stereographic}
w \longmapsto 
\frac{1}{1-\sum_j w_j}(-\textstyle\sum_{j \neq 0} w_j,w_1,\dots). 
\end{equation}
The quotient map $q|B: B \rightarrow \bC P^\infty$ maps $B \setminus \partial B$ isomorphically to $\bC P^\infty \setminus H$, and collapses the boundary $\partial B = S$ onto $H$. 
(In the analogous finite-dimensional situation, $q|B$ describes how complex projective space is obtained from its hypersurface $H$ by attaching a cell.)

We will use the Morse function
\begin{equation} \label{eq:morse}
\begin{aligned}
& h: \bC P^\infty \longrightarrow \bR, \\
& h(w) = |w_1|^2 + 2|w_2|^2 + 3|w_3|^2 + \cdots
\end{aligned}
\end{equation}
and the standard (Fubini-Study) metric. The critical points are precisely the $c_k$, and they have Morse index $2k$. The negative gradient flow is the projectivization of the linear flow
\begin{equation} \label{eq:linear-flow}
s \cdot (w_0,w_1,w_2,\dots) = (w_0,e^{-2s}w_1,e^{-4s}w_2,\dots).
\end{equation}
The stable and unstable manifolds are
\begin{equation} \label{eq:stable-unstable}
\begin{aligned}
& W^s(c_k) = \{w_0 = \cdots = w_{k-1} = 0\}, \\
& W^u(c_k) = \{w_{k+1} = w_{k+2} = \cdots = 0\}.
\end{aligned}
\end{equation}
Those manifolds intersect transversally, making the flow Morse-Smale. Moreover, they are also transverse to \eqref{eq:h-hypersurface}.

We also want to fix a connection $A$ on the circle bundle \eqref{eq:q-bundle}. This must be invariant under the shift, and flat in a neighbourhood of each $c_k$. Every path joining two critical points yields a parallel transport map, which in view of \eqref{eq:ck-identification} can be thought of as an element of $S^1$. 

\subsection{Spaces of trajectories\label{subsec:trajectories}}
All our spaces are defined as standard compactifications (by broken trajectories) of suitable spaces of negative gradient trajectories for the function \eqref{eq:morse}. Concretely:

\begin{itemize}
\itemsep.5em
\item
For $k > 0$, consider the space of unparametrized trajectories going from $c_k$ to $c_0$ (using \eqref{eq:shift-embedding}, one can identify this with the space of trajectories from $c_{k+l}$ to $c_l$, for any $l$).  Denote the standard compactification of the trajectory space by $P_k$. This is a $(2k-1)$-dimensional smooth compact manifold with corners, and comes with a canonical identification (which describes its boundary as the union of codimension $1$ closed boundary faces)
\begin{equation} \label{eq:boundary-pk-d}
\partial P_k \; \iso \; \bigcup_{\!\!k_1+k_2 = k} P_{k_1} \times P_{k_2}.
\end{equation}
We will denote unparametrized trajectories by $[v]$, thinking of them as equivalence classes under the action of $\bR$. Points in the interior of a boundary face \eqref{eq:boundary-pk-d} correspond to two-component broken flow lines $([v_1],[v_2])$.

\item
A closely related space is the compactification of the space of parametrized trajectories, denoted by $P_k^p$ for $k \geq 0$. This has dimension $2k$, and satisfies
\begin{equation} \label{eq:boundary-pk-p}
\partial P_k^{p} \;=\; \Big( \bigcup_{k_1+k_2 = k} P_{k_1} \times P_{k_2}^{p} \Big)
\cup \Big( \bigcup_{k_1+k_2 = k} P_{k_1}^p \times P_{k_2} \Big).
\end{equation}

\item
Consider pairs $(w,v)$, where $v$ is a trajectory and $w \in S^\infty$ a point such that $q(w) = v(0)$. Such pairs form a circle bundle over the space of parametrized trajectories. There is also a compactification $P_k^s$, of dimension $2k+1$, which is a circle bundle over $P_k^b$, satisfying the obvious analogue of \eqref{eq:boundary-pk-p}:
\begin{equation} \label{eq:boundary-pk-s}
\partial P_k^s \;=\; \Big( \bigcup_{k_1+k_2 = k} P_{k_1} \times P_{k_2}^s \Big)
\cup \Big( \bigcup_{k_1+k_2 = k} P_{k_1}^s \times P_{k_2} \Big).
\end{equation}

\item
Finally, one could modify the most recent definition by allowing $w \in B^\infty$ to be a point lying on the line singled out by $v(0)$. This gives a disc bundle whose boundary is our previous circle bundle. The compactification $P_k^b$ has dimension $2k+2$, and satisfies 
\begin{equation} \label{eq:boundary-pk-b}
\partial P_k^b = P_k^s \cup \Big( \bigcup_{k_1+k_2 = k} P_{k_1} \times P_{k_2}^b \Big)
\cup \Big( \bigcup_{k_2+k_1 = k} P_{k_1}^b \times P_{k_2} \Big).
\end{equation}
\end{itemize}
%
The reader will have noticed that we have, without further ado, declared our compactified moduli spaces to be smooth manifolds with corners, in a way which is compatible with the product structure on boundary strata. Such smooth structures are constructed in \cite{latour94, burghelea-haller01, wehrheim12}, under the assumption that there is a local chart around each critical point, in which the Morse function and the metric are both standard. While our metric does not satisfy that condition, there are local charts around the critical points in which the gradient flow is linear, see \eqref{eq:linear-flow}; and that is sufficient to make the constructions go through. Alternatively, since our function and gradient flow are completely explicit, one could construct the necessary charts near the boundary strata by hand. 

Over each of our spaces of trajectories, there is a ``tautological family''
\begin{equation} \label{eq:tautological}
\begin{aligned}
& \scrP_k \longrightarrow P_k, \\
& \scrP_k^p \longrightarrow P_k^p, \\
& \scrP_k^s \longrightarrow P_k^s, \\
& \scrP_k^b \longrightarrow P_k^b.
\end{aligned}
\end{equation}
Let's consider the first case: 
\begin{itemize}
\item A point of $\scrP_k$ is represented by a (possibly broken) flow line with additional data:
\begin{equation} \label{eq:tauto-explicit}
(v_1,\dots,v_j,\scrl,\scrs) \quad \text{for some $j \geq 1$, $\scrl \in \{1,\dots,j\}$, and $\scrs \in \bR$.} 
\end{equation}
We identify two representatives iff they are related by the action of $(r_1,\dots,r_j) \in \bR^j$:
\begin{equation} \label{eq:tauto-explicit-2}
(v_1,\dots,v_j,\scrl,\scrs) \sim (v_1(\cdot+r_1),\dots,v_j(\cdot+r_j),\scrl,\scrs-r_\scrl).
\end{equation} 
\end{itemize}
$\scrP_k$ is a noncompact manifold with corners, carrying a free and proper $\bR$-action (by translation on $\scrs$); the map to $P_k$, which forgets $(\scrl,\scrs)$, is invariant under that action. In fact, the fibre of the map to $P_k$ over any broken trajectory with $j$ components can be identified with a disjoint union of $j$ copies of the real line (ordered in a preferred way), with $\bR$ acting by translation on each. This description is compatible with \eqref{eq:boundary-pk-d}, meaning that the restriction of $\scrP_k$ to $P_{k_1} \times P_{k_2} \subset \partial P_k$ is canonically identified with the disjoint union of the pullbacks of $\scrP_{k_1}$ and $\scrP_{k_2}$. Additionally, $\scrP_k$ comes with a smooth evaluation map to $\bC P^\infty$, which takes $(v_1,\dots,v_j,\scrl,\scrs)$ to $v_\scrl(\scrs)$; this intertwines the $\bR$-action with the negative gradient flow \eqref{eq:linear-flow}.
 
We also find it convenient to introduce a compactification $\bar\scrP_k$ (just as a topological space, without differentiable structure) by allowing the point on our flow line to degenerate. In the notation from \eqref{eq:tauto-explicit}, we now allow $\scrs = \pm\infty$, but additionally identify
\begin{equation}
(v_1,\dots,v_j,\scrl,+\infty) \htp (v_1,\dots,v_j,\scrl+1,-\infty).
\end{equation}
The map from \eqref{eq:tautological} extends to $\bar\scrP_k \rightarrow P_k$. The fibre of the extended map over a broken trajectory with $j$ components consists of $j$ copies of $\bar{\bR} = \bR \cup \{\pm\infty\}$, with the $+\infty$ point of each glued to the $-\infty$ point of the following one (so that overall, one gets a space homeomorphic to a closed interval). The $\bR$-action extends to a continuous action on $\bar\scrP_k$, which leaves $\bar\scrP_k \setminus \scrP_k$ fixed. The evaluation map to $\bC P^\infty$ extends continuously to $\bar\scrP_k$, taking $[v_1,\dots,v_j,\scrl,\pm\infty]$ to the critical point which is the $s \rightarrow \pm \infty$ limit of $v_\scrl$. Moreover, there are canonical continuous sections which single out the endpoints of the chain of $\bar{\bR}$'s:
\begin{equation} \label{eq:endpoint-sections}
\begin{aligned}
& y_{\pm}: P_k \longrightarrow \bar{\scrP}_k \setminus \scrP_k, \\
& y_-([v_1,\dots,v_j]) = [v_1,\dots,v_j,1,-\infty], \\
& y_+([v_1,\dots,v_j]) = [v_1,\dots,v_j,j,+\infty].
\end{aligned}
\end{equation}
The compactifications $\bar\scrP_k$ are compatible with \eqref{eq:boundary-pk-d}, in a sense which is similar to our previous statement of the same kind, and which we will therefore not spell out. 

The next case in \eqref{eq:tautological} is an appropriate modification of the previous construction:
\begin{itemize}
\item
A point of $\scrP_k^p$ is represented by
\begin{equation}
(v_1,\dots,v_j,i,\scrl,\scrs) \text{for some $j \geq 1$, $i,\scrl \in \{1,\dots,j\}$, and $\scrs \in \bR$.}
\end{equation}
Here, the component $v_i$ may be a constant flow line. We divide out by $\bR^{i-1} \times \{0\} \times \bR^{j-i} \subset \bR^j$, acting as in \eqref{eq:tauto-explicit-2}.
\end{itemize}
A fibre of the map $\scrP_k^p \rightarrow P_k^p$ over a trajectory with $j$ components again consists of $j$ copies of $\bR$, even if the parametrized component is constant. The space $\scrP_k^p$ comes with the same $\bR$-action as before. Additionally, there is a distinguished smooth section
\begin{equation} \label{eq:canonical-section}
\begin{aligned}
& y_*: P_k^p \longrightarrow \scrP_k^p, \\
& y_*([v_1,\dots,v_j,i]) = [v_1,\dots,v_j,i,i,0].
\end{aligned}
\end{equation}
There is also a compactification $\bar\scrP_k^p$, with additional sections as in \eqref{eq:endpoint-sections}. The other two cases in \eqref{eq:tautological} are parallel.

\begin{remark}
The reader may have noticed that $\bar\scrP_k$ is homeomorphic to $P_k^p$, hence after all does carry the structure of a smooth manifold with corners. The same is true for the other compactified moduli spaces, which can all be thought of as moduli spaces of broken trajectories with one marked point (which can lie on an additional constant component). However, those smooth structures will be irrelevant for our purpose.
\end{remark}

\subsection{Topological aspects\label{subsec:topology}}
In low-dimensional cases, the topology of the moduli spaces of trajectories is easy to determine: there are diffeomorphisms
\begin{equation} \label{eq:low}
\begin{aligned}
& && P_1 \iso S^1, && P_2 \iso S^1 \times D^2, \\
& P_0^p \iso \mathit{point}, && P_1^p \iso S^1 \times [0,1], \\
& P_0^s \iso S^1, && P_1^s \iso S^1 \times S^1 \times [0,1], \\
& P_0^b \iso D^2.
\end{aligned}
\end{equation}
In two of those cases, we want to fix choices of diffeomorphisms, which will be used in orientation arguments later on. For $P_0^s$, we take the obvious identification $S^1 \iso q^{-1}(c_0) = P_0^s$, which was spelled out in \eqref{eq:ck-identification}. For $P_1$, we choose 
\begin{equation} \label{eq:p1}
\begin{aligned}
& S^1 \stackrel{\iso}{\longrightarrow} P_1, \\
& r \longmapsto [s \mapsto (e^{2\pi i r}:e^{-2s}:0:\cdots)] = [s \mapsto (1:e^{-2s-2\pi i r}:0:\cdots)].
\end{aligned}
\end{equation}
We also need some more topological information about the boundary strata of the low-dimensional spaces \eqref{eq:low}.

\begin{lemma} \label{th:112}
Consider the map induced by \eqref{eq:boundary-pk-d} for $k = 2$,
\begin{equation} \label{eq:112}
H_1(P_1) \oplus H_1(P_1) = H_1(P_1 \times P_1) \longrightarrow H_1(P_2)
\end{equation}
(the domain and target are isomorphic to $\bZ^2$ and $\bZ$, respectively; and we know that the map is onto). This map is diagonal, meaning that it is invariant under switching the two $P_1$ factors.
\end{lemma}

\begin{proof}
Take the action of $(S^1)^3$ on $\bC P^\infty$ which rotates the first three coordinates. This induces an action on $P_2$, for which the diagonal subgroup acts trivially, and the orbits of the subgroup $(1,e^{2\pi i r}, 1)$ are contractible (since some of those orbits are fixed points, and any orbit can be deformed to one of them). If we let the same group act on the boundary \eqref{eq:112}, that action has weights $(0,1,-1)$ on the first $P_1$ factor, and weights $(1,-1,0)$ on the second $P_1$ factor (because those factors correspond to flow lines lying in $\{0\} \times \bC P^1 \times \{0,\dots,\}$ and $\bC P^1 \times \{0,\dots,\}$, respectively). In particular, the subgroup $(1,e^{2\pi i r},1)$ acts with weights $1$ and $-1$ on the two boundary factors. By taking an orbit of that subgroup, and moving it from the boundary to the interior, one sees that the element $(1,-1) \in H_1(P_1) \oplus H_1(P_1)$ lies in the kernel of \eqref{eq:112}.
\end{proof}

\begin{lemma} \label{th:01-s}
Consider the maps induced by \eqref{eq:boundary-pk-s} for $k = 1$,
\begin{equation} \label{eq:011-inclusions}
\begin{aligned}
& H_1(P_0^s) \oplus H_1(P_1) = H_1(P_0^s \times P_1) \longrightarrow H_1(P_1^s), \\
& H_1(P_1) \oplus H_1(P_0^s) = H_1(P_1 \times P_0^s) \longrightarrow H_1(P_1^s). 
\end{aligned}
\end{equation}
All groups involved are isomorphic to $\bZ^2$, and the maps are isomorphisms. Composing the first map in \eqref{eq:011-inclusions} with the inverse of the second map yields an element of $\mathit{GL}_2(\bZ)$ which, with respect to the bases determined by our fixed identifications, is given by 
\begin{equation} \label{eq:invmatrix}
\begin{pmatrix} 0 & 1 \\ 1 & 1 \end{pmatrix}.
\end{equation}
\end{lemma}

\begin{proof}
Take the action of $S^1 \times S^1$ on $\bC^\infty$ which rotates the first two coordinates. This induces an action on $P_1^s$, each orbit of which is homotopy equivalent to the whole space. If we restrict the action to $P_0^s \times P_1 \subset \partial P_1^s$, the action on $P_0^s$ has weights $(0,1)$, while that on $P_1$ has weights $(1,-1)$ (for the same reason as in Lemma \ref{th:112}). On the other boundary component $P_1 \times P_0^s$, we still get weights $(1,-1)$ on the $P_1$ factor, but weights $(1,0)$ on the $P_0^s$ factor. In other words, the maps \eqref{eq:011-inclusions} sit in a commutative diagram of isomorphisms (in which the desired map \eqref{eq:invmatrix} sits as the dashed arrow)
\begin{equation}
\xymatrix{
& H_1(S^1 \times S^1) \ar[dl]_{\left(\begin{smallmatrix} 0 & 1 \\ 1 & -1 \end{smallmatrix}\right)} 
\ar[dr]^{\left(\begin{smallmatrix} 1 & -1 \\ 1 & 0 \end{smallmatrix}\right)} &
\\
\ar@{-->}[rr]
\ar[dr] H_1(P_0^s \times P_1) && H_1(P_1 \times P_0^s) \ar[dl]
\\
& H_1(P_1^s).
}
\end{equation}
\end{proof}

The topology of the higher-dimensional spaces of trajectories is as follows:

\begin{lemma} \label{th:s-cobordism}
Up to homeomorphism, one has, for $k>0$,
\begin{equation}
\begin{aligned}
& P_k \iso S^1 \times D^{2k-2}, \\
& P_k^p \iso S^1 \times D^{2k-1}, \\
& P_k^s \iso S^1 \times S^1 \times D^{2k-1}, \\
& P_k^b \iso S^1 \times D^{2k+1}.
\end{aligned}
\end{equation}
\end{lemma}

\begin{proof}
Let's start with the most basic situation, which is that of the $P_k^p$. The cases $k = 1,2$ can be dealt with by hand (the first is already in \eqref{eq:low}, and the second is easy, since we don't care about the differentiable structure here). We therefore assume that $\mathrm{dim}(P_k^p) \geq 6$. From \eqref{eq:stable-unstable} one sees that
\begin{equation}
P_k^p \setminus \partial P_k^p \iso \bC^* \times \bC^{k-1} \iso S^1 \times \bR^{2k-1}.
\end{equation}
Now consider $P_k^p$ itself, but with the corners smoothed, so that it is a compact manifold with boundary. After removing a large piece of the interior, one ends up with a cobordism between $S^1 \times S^{2k-1}$ and $\partial P_k^p$. The $s$-cobordism theorem (for manifolds with free abelian fundamental group) implies that this is trivial. Reattaching the piece we had removed yields the desired result.

By construction, $P_k^s$ is a circle bundle over $P_k^p$, and because of the topology of the latter space, that circle bundle is necessarily trivial. Similarly, $P_k^b$ is the disc bundle associated to that circle bundle, hence homeomorphic to $P_k^p \times D^2$. For the spaces $P_k$, one needs to repeat the previous $h$-cobordism argument (leaving one more case not in \eqref{eq:low}, namely $P_3$, to be dealt with by hand).
\end{proof}

The main way in which the topology of the moduli spaces enters into our discussion is through certain circle-valued maps. We will give two constructions of such maps: a direct geometric one, and another one by a topological argument.

\begin{lemma} \label{th:alpha}
There are smooth maps 
\begin{equation} \label{eq:alpha}
\alpha_k: P_k \longrightarrow S^1,
\end{equation}
such that: $\alpha_1$ has degree $1$; and the maps are compatible with \eqref{eq:boundary-pk-d}, in the sense that
\begin{equation} \label{eq:alpha-additivity}
\alpha_k\, |\, (P_{k_1} \times P_{k_2}) = \alpha_{k_1} + \alpha_{k_2}.
\end{equation}
\end{lemma}

\begin{proof}[First proof]
Given a gradient trajectory $v$ from $c_k$ to $c_0$, let's use parallel transport (for the connection $A$) to get a map $q^{-1}(c_k) \rightarrow q^{-1}(c_0)$. The parallel transport map is given by an element of $S^1$, and we set $\alpha_k([v])$ to be that element. The desired property for $k = 1$ can be shown, for instance, by deforming our connection $A$ to the standard round connection, for which the parallel transport maps exactly recover the identification \eqref{eq:p1}. Parallel transport maps extend smoothly to broken flow lines (this is easy to see since $A$ is flat near the critical points); and the equality \eqref{eq:alpha-additivity} is just their basic concatenation property.
\end{proof}

\begin{proof}[Second proof]
Choose an arbitrary $\alpha_1$ with the desired property, and consider the map 
\begin{equation}
\begin{aligned}
& \partial P_2 = P_1 \times P_1 \rightarrow S^1, \\
& ([v_1],[v_2]) \longrightarrow \alpha_1([v_1]) + \alpha_1([v_2]). 
\end{aligned}
\end{equation}
Lemma \ref{th:112} (or rather, the dual statement for cohomology) shows that this can be extended to $\alpha_2$. From now on, one proceeds inductively as follows. Suppose that, for some $l \geq 3$, we have already defined $\alpha_1,\dots,\alpha_{l-1}$ with the desired properties. The requirement \eqref{eq:alpha-additivity} then prescribed the value of $\alpha_l$ on $\partial P_l$. By Lemma \ref{th:s-cobordism}, the pair $(P_l, \partial P_l)$ is $2$-connected as soon as $l \geq 3$. Hence, any circle-valued map can be extended from $\partial P_l$ to $P_l$.
\end{proof}

\begin{lemma} \label{th:alphabeta-s}
There are smooth maps 
\begin{equation}
\alpha_k^s,\, \beta_k^s: P_k^s \longrightarrow S^1,
\end{equation}
such that: $\alpha_0^s$ is zero, while $\beta_0^s$ has degree $1$; and the restriction to \eqref{eq:boundary-pk-s} is given by
\begin{equation} \label{eq:left-alphabeta}
\left\{
\begin{aligned} 
& \alpha_k^s \,|\, (P_{k_1} \times P_{k_2}^s) = \alpha_{k_1} + \alpha_{k_2}^s, \\
& \beta_k^s \,|\, (P_{k_1} \times P_{k_2}^s) = -\alpha_{k_1} + \beta_{k_2}^s,
\end{aligned}
\right.
\end{equation}
and
\begin{equation}
\left\{
\begin{aligned} \label{eq:right-alphabeta}
& \alpha_k^s \,|\, (P_{k_1}^s \times P_{k_2}) = \alpha_{k_1}^s + \alpha_{k_2}, \\
& \beta_k^s \,|\, (P_{k_1} ^s \times P_{k_2}) = \beta_{k_1}^s.
\end{aligned}
\right.
\end{equation}
\end{lemma}

\begin{proof}[First proof]
One can define $\alpha_k^s$ exactly as before, by parallel transport along $v$. Similarly, consider inverse parallel transport along $v|(-\infty,0]$, which yields a map $q^{-1}(v(0)) \rightarrow q^{-1}(c_k) \iso S^1$. We define $\beta_k^s(w,v)$ to be the image of $w$ under that map. The required properties are obvious. 
\end{proof}

\begin{proof}[Second proof]
Choose the maps first for $k = 0$. Then $(\alpha_1^s,\beta_1^s)$ is supposed to be a map whose restriction to the two boundary components of $P_1^s$ induces the following maps on homology:
\begin{equation}
\begin{pmatrix} 1 & 0 \\ -1 & 1 \end{pmatrix} \text{ for $P_1 \times P_0^s$,} \quad \text{ and } \quad
\begin{pmatrix} 0 & 1 \\ 1 & 0 \end{pmatrix} \text{ for $P_0^s \times P_1$.}
\end{equation}
Lemma \ref{th:01-s} shows that such maps exist. It is then easy to adjust them so that \eqref{eq:left-alphabeta} and \eqref{eq:right-alphabeta} are satisfied. As in Lemma \ref{th:alpha}, the rest of the construction is on autopilot: by Lemma \ref{th:s-cobordism}, the pair $(P_l^s,\partial P_l^s)$ is $2$-connected as soon as $l \geq 2$. Hence, any map $\partial P_l^s \rightarrow S^1 \times S^1$ extends to $P_l$.
\end{proof}

Finally, some orientation considerations will be needed. The interior of $P_k^p$ can be thought of as a locally closed complex submanifold of $\bC P^\infty$, and we orient it in the standard way. Again at an interior point $[v]$, the space $P_k$ comes with a short exact sequence
\begin{equation} \label{eq:r-quotient}
0 \rightarrow \bR \partial_s v\longrightarrow T_v P_k^p \longrightarrow T_{[v]} P_k \rightarrow 0.
\end{equation}
We choose our orientation of $P_k$ so that, for a splitting $T_v P_k^p \iso \bR  \oplus T_{[v]} P_k$ of \eqref{eq:r-quotient}, it is compatible with the orientation of $P_k^p$. For $P_k^s$ and $P_k^b$, we use a similar strategy, based on the long exact sequences
\begin{align}
\label{eq:circle-orientation-sequence}
& 
0 \rightarrow \bR (iw,0) \longrightarrow T_{(w,v)} P_k^s \longrightarrow T_v P_k^p \rightarrow 0, 
\\
&
0 \rightarrow \bC (v(0),0) \longrightarrow T_{(w,v)} P_k^b \longrightarrow T_v P_k^p \rightarrow 0.
\end{align}
As an example, consider $P_1$. In \eqref{eq:p1}, $(\partial_s v, \partial_r v)$ is a positively oriented basis of $T_v P_1^p$; hence, that parametrization is compatible with our overall choice of orientations. Likewise, the orientation coming from \eqref{eq:circle-orientation-sequence} is compatible with the identification $P_0^s \iso S^1$.

\begin{lemma} \label{th:orientations}
(i) The orientations of $P_k$ are compatible with \eqref{eq:boundary-pk-d}. This means that the boundary orientation induced by that of $P_k$ agrees with the product orientation of $P_{k_1} \times P_{k_2}$.
%
%
%

(ii) The orientations of $P_k^p$ are compatible with the product orientations of the boundary faces $P_{k_1} \times P_{k_2}^p$, while for faces of the form $P_{k_1}^p \times P_{k_2}$ the orientations are opposite.
%

(iii) The orientations of $P_k^s$ are compatible with \eqref{eq:boundary-pk-s}.
%

(iv) The orientations of $P_k^b$ are compatible with the product orientations of the boundary face $P_k^s$. The same holds for boundary faces $P_{k_1} \times P_{k_2}^b$, while for those of the form $P_{k_1}^b \times P_{k_2}$ the orientations are opposite.
%
\end{lemma}

\begin{proof}
We find it convenient to temporarily introduce another space $P_k^t$, which is the compactification of the space of flow lines $v$ equipped with two marked points $s_1 < s_2$. More precisely, we divide by the common $\bR$-action, so points in the interior of $P_k^t$ are equivalence classes $[s_1,s_2,v]$. Among the boundary strata of this space are
\begin{equation} \label{eq:t-boundary}
\bigcup_{k_1+k_2 = k} (P_{k_1}^p \setminus \partial P_{k_1}^p) \times (P_{k_2}^p \setminus \partial P_{k_2}^p) \subset \partial P_k^t,
\end{equation}
where one thinks of the boundary points as trajectories broken into two pieces, each of them carrying one marked point, which fixes the parametrization. Let's suppose that we have oriented $P_k^t$ by mapping (on the interior) $[s_1,s_2,v] \longmapsto (s_2-s_1,v(\cdot + s_1)) \in\bR \times P_k^p$, and using the complex orientation of (the interior of) $P_k^p$. Then, it is easy to see that \eqref{eq:t-boundary} is compatible with orientations.

Consider the $\bR^2$-action on $P_k^t$ by moving the two marked points. On the interior of the moduli space, this is given by 
$
(r_1,r_2) \cdot [0,s_2,v] = [r_1,s_2+r_2,v] = [0,s_2+r_2-r_1,v(\cdot+r_1)] 
$.
Assuming that $v$ is not constant, one gets a short exact sequence
\begin{equation} \label{eq:t-versus-nothing}
0 \rightarrow \bR \oplus \bR \longrightarrow T_{[0,s_1,v]}P_k^t \longrightarrow T_{[v]}P_k \rightarrow 0,
\end{equation}
where the first map takes the standard generators of $\bR^2$ to $(0,-1,\partial_s v)$ and $(0,1,0)$. In order for the resulting splitting
\begin{equation} \label{eq:t-splitting}
T_{[0,s_1,v]} P_k^t \iso \bR \oplus \bR \oplus T_{[v]}P_k
\end{equation}
to compatible with the chosen orientations, the two $\bR$ summands would have to appear in the opposite order; hence, \eqref{eq:t-versus-nothing} is incompatible with orientations. In the limit where $[s_1,s_2,v]$ degenerates to a point $(v_1,v_2)$ in a boundary stratum \eqref{eq:t-boundary}, the $\bR^2$-action becomes the reparametrization action on both factors. Again assuming that neither flow line is constant, we have another short exact sequence,
\begin{equation} \label{eq:prod-sequence}
0 \rightarrow \bR \oplus \bR \longrightarrow T_{v_1}P_{k_1}^p \oplus T_{v_2}P_{k_2}^p \longrightarrow T_{[v_1]}P_{k_1} \oplus T_{[v_2]}P_{k_2} \rightarrow 0,
\end{equation}
where the first map has image generated by $(\partial_{s_1} v_1,0)$ and $(0,\partial_{s_2}v_2)$. In order
for the resulting splitting
\begin{equation} \label{eq:t-boundary-splitting}
T_{v_1}P_{k_1}^p \oplus T_{v_2}P_{k_2}^p \iso \bR \oplus \bR \oplus T_{[v_1]}P_{k_1} \oplus T_{[v_2]}P_{k_2}
\end{equation}
to be compatible with the chosen orientation of $P_{k_1}$ and $P_{k_2}$, the second and third summands in \eqref{eq:t-boundary-splitting} would have to swap positions, leading to a Koszul sign
$(-1)^{\mathrm{dim}(P_{k_1})} = -1$, which agrees with that in \eqref{eq:t-splitting}. Using our previous observation about orientations in \eqref{eq:t-boundary}, we can now obtain (i).

We derive (ii) by a similar argument. Take a point $(v_1,[v_2]) \in P_{k_1}^p \times P_{k_2} \subset \partial P_k^p$, and suppose for simplicity that $v_1$ is not constant (to deal with the constant case, one would have to go back to the spaces $P_k^t$). The $\bR$-action by reparametrization yields an analogue of \eqref{eq:t-boundary-splitting},
\begin{equation}
\label{eq:p-boundary-splitting}
T_{v_1}P_{k_1}^p \oplus T_{[v_2]} P_{k_2} \iso \bR \oplus T_{[v_1]} P_{k_1} \oplus T_{[v_2]} P_{k_2},
\end{equation}
which is compatible with orientations. In the parallel case with $([v_1],v_2) \in P_{k_1} \times P_{k_2}^p$, one acquires a Koszul sign $(-1)^{\mathrm{dim}(P_{k_1})} = -1$. This explains the sign difference between the two kinds of boundary faces of $P_k^p$. To get the correct result, we need one more observation: the operations of dividing by an $\bR$-action, and passing to the boundary, don't commute in their effect on orientations (quotienting by $\bR$ and then passing to the boundary yields the opposite orientation of first passing to the boundary and then quotienting by $\bR$).

The proof of (iii) is similar to that of (i), and that of (iv) similar to that of (ii).
\end{proof}

In fact, only the spaces $P_k$ and $P_k^s$ will play a significant role in our application. We have included $P_k^p$ since it appears as an obvious intermediate step in the discussion; and $P_k^b$ because another, similarly defined, space will be important in the next section.

\subsection{Trajectories with evaluation constraints}
Spaces of trajectories going through a fixed submanifold are a well-known concept, usually arising in the definition of the cap product on Morse homology (see e.g.\ \cite[p.\ 177]{donaldson02b}). 
We will use the following specific instances.
\begin{itemize}
\itemsep.5em
\item
Let $Q_k^p$ be the compactification of the space consisting of trajectories $v$, with the usual limits, such that $v(0)$ lies on the hypersurface $H$ from \eqref{eq:h-hypersurface}. This is of dimension $2k-2$, and satisfies 
\begin{equation} \label{eq:boundary-pk-p2}
\partial Q_k^p = \Big( \bigcup_{k_1+k_2 = k} P_{k_1} \times Q_{k_2}^p \Big)
\cup \Big( \bigcup_{k_1+k_2 = k} Q_{k_1}^p \times P_{k_2} \Big).
\end{equation}
This description of the boundary strata relies on the compatibility of $H$ with the coordinate shift map \eqref{eq:shift-embedding}; more precisely, we use the fact that the intersection $H \cap \{w_0 = 0\}$ is the image of $H$ under the shift.

\item
Consider the space of pairs $(w,v)$, where $v$ is a trajectory and $w \in S = q^{-1}(H)$ a point such that $q(w) = v(0)$. This is a circle bundle over the previous space. Its compactification $Q_k^s$ has dimension $2k-1$, and satisfies
\begin{equation} \label{eq:boundary-pk-s2}
\partial Q_k^s = \Big( \bigcup_{k_1+k_2 = k} P_{k_1} \times Q_{k_2}^s \Big)
\cup \Big( \bigcup_{k_1+k_2 = k} Q_{k_1}^s \times P_{k_2} \Big).
\end{equation}

\item
One obtains spaces $Q_k^b$ by instead allowing any point $w \in B$, for $B$ as in \eqref{eq:b-bounding}, which lies on the complex line determined by $v(0)$. The resulting spaces are $2k$-dimensional, and satisfy
\begin{equation} \label{eq:boundary-pk-b2}
\partial Q_k^b = Q_k^s \cup \Big( \bigcup_{k_1+k_2 = k} P_{k_1} \times Q_{k_2}^b \Big)
\cup \Big( \bigcup_{k_1+k_2 = k} Q_{k_1}^b \times P_{k_2} \Big).
\end{equation}
\end{itemize}

In low-dimensional instances,
\begin{equation} \label{eq:low-dim-q}
\begin{aligned}
& && Q_1^p \iso \mathit{point}, \\
& && Q_1^s \iso S^1, \\
& Q_0^b \iso \mathit{point},
&& Q_1^b \iso \text{\it pair-of-pants}.
\end{aligned}
\end{equation}
Explicitly, the unique point of $Q_0^b$ consists of the constant trajectory $v \equiv [1:0:\cdots] \in \bC P^\infty$ together with the point $w = (-1,0,\dots) \in B$. The condition for trajectories in the interior of $Q_1^b$ is that $v(0)$ should lie in $\bC^* \subset \bC P^1$, and should come with a $w = (w_0,w_1,0,\dots) \in B \setminus \partial B$ such that $w_1/w_0 = v(0)$. There is one such $w$ for any $v(0)$, with the exception of $v(0) = -1$, where one would necessarily have $w \in \partial B$. Hence
\begin{equation} \label{eq:3-punctured}
Q_1^b \setminus \partial Q_1^b = \bC^* \setminus \{-1\}
\end{equation}
is a three-punctured sphere, which implies the statement about the compactification made in \eqref{eq:low-dim-q}.

\begin{lemma} \label{th:break-1}
There is a sequence of compact manifolds with corners $R_k^p$, $k>0$, satisfying
\begin{equation} \label{eq:break-1}
\partial R_k^p = Q_k^p \cup P_{k-1}^p \cup \Big( \bigcup_{k_1+k_2 = k} P_{k_1} \times R_{k_2}^p
\Big) \cup \Big( \bigcup_{k_1+k_2 = k} R_{k_1}^p \times P_{k_2} \Big),
\end{equation}
and these identifications of boundary strata are compatible with \eqref{eq:boundary-pk-p} and \eqref{eq:boundary-pk-p2}.
\end{lemma}

\begin{proof} 
Consider pairs of half-flow-lines $(v^-,v^+)$:
\begin{equation} \label{eq:half-trajectories}
\left\{
\begin{aligned} 
& \textstyle v^-: (-\infty,0] \longrightarrow \bC P^\infty, \;\; \partial_s v^- + \nabla h = 0, &&
\textstyle\lim_{s \rightarrow -\infty} v^-(s) = c_k, \;\; v^-(0) = x^-,
\\ 
& \textstyle v^+: [0,\infty) \longrightarrow \bC P^\infty, \;\; \partial_s v^+ + \nabla h = 0, &&
\textstyle\lim_{s \rightarrow +\infty} v^+(s) = c_0, \quad v^+(0) = x^+.
\end{aligned}
\right.
\end{equation}
The endpoints $(x^-,x^+)$ can be any points in $\bC P^\infty$ satisfying
\begin{equation} \label{eq:added}
\left\{
\begin{aligned}
& x_k^- \neq 0, \\ & x_{k+1}^- = x_{k+2}^- = \cdots = 0, \\ & x_0^+ \neq 0.
\end{aligned}
\right.
\end{equation}
One obtains the interior of $Q_k^p$ by additionally imposing the coincidence conditions
\begin{equation} \label{eq:match-1}
\left\{
\begin{aligned}
& x^-_0 + x^-_1 + \cdots = 0, \\ 
& x^-_0 = x^+_0, \\
& x^-_1 = x^+_1, \\
& \dots
\end{aligned}
\right.
\end{equation}
Originally, these equations took place in $\bC P^\infty$ so we should say that $x_j^- = \lambda x_j^+$ for some $\lambda \in \bC^*$. However, it is notationally a bit simpler to ask for $\lambda = 1$, and correspondingly consider the $x^{\pm}$ up to rotation by a common factor.

Let's introduce a parameter $\theta \in \bC$, and deform the conditions in \eqref{eq:match-1} as follows:
\begin{equation} \label{eq:match-2}
\left\{
\begin{aligned}
& x^-_0 + x^-_1 + \cdots = 0, \\
& x^-_0 = x_0^+, \\
& x^-_1 = x_1^+ - \theta x_0^+, \\
& x^-_2 = x_2^+ - \theta x_1^+, \\
& \dots
\end{aligned}
\right.
\end{equation}
Note that if one sets $x_0^+ = 0$ in these equations, it follows that $x_0^- = 0$ as well, and then the equations for the remaining $x_j^{\pm}$ reproduce the original ones after an index shift $j \mapsto j-1$. There is a minor issue here, which becomes evident when combining \eqref{eq:match-2} with the convergence conditions \eqref{eq:added}: for general $\theta$, we have to allow the point $x^+$ to lie outside $\bC P^\infty$, since it satisfies $x_j^+ = \theta x_{j-1}^+$ for all $j>k$. This means that the solutions $v^+$ also lie outside $\bC P^\infty$. In practice, this is unproblematic: it is still possible to write the combined conditions in terms of finitely many variables $(x_0^{\pm},\dots,x_k^{\pm})$, as 
\begin{equation} \label{eq:k-equation}
\left\{
\begin{aligned}
& x_0^- + \cdots + x_k^- = 0, \\
& x^-_0 = x_0^+, \\
& x^-_1 = x_1^+ - \theta x_0^+, \\
& x^-_2 = x_2^+ - \theta x_1^+, \\
& \dots \\
& x^-_k = x_k^+ - \theta x_{k-1}^+, \\
& x_0^+ \neq 0, \\
& x^-_k \neq 0.
\end{aligned}
\right.
\end{equation}
Define $R_k^{p,\theta}$ to be the compactification of the space of solutions of \eqref{eq:half-trajectories} and \eqref{eq:match-2} by broken trajectories. This is smooth for any $\theta$, and satisfies
\begin{equation} \label{eq:boundary-theta}
\partial R_k^{p,\theta} = \Big( \bigcup_{k_1+k_2 = k} P_{k_1} \times R_{k_2}^{p,\theta} \Big)
\cup \Big( \bigcup_{k_1+k_2 = k} R_{k_1}^{p,\theta} \times P_{k_2} \Big).
\end{equation}
If we set $\theta = 0$, \eqref{eq:match-2} reduces to the original \eqref{eq:match-1}, and correspondingly $R_k^{p,\theta=0} = Q_k^p$. On the other hand, for $\theta = 1$ the sum of all equations in \eqref{eq:k-equation} says that $x_k^+ = 0$, hence $x_j^+ = 0$ for all $j > k$ and we land back in $\bC P^\infty$.

We want to introduce another parameter-dependent set of coincidence conditions:
\begin{equation} \label{eq:match-3}
\left\{
\begin{aligned}
& x_0^- + \eta x_1^- + \eta^2 x_2^- + \cdots = 0, \\
& x_1^-  + \eta x_2^- + \eta^2 x_3^- + \cdots = -x_0^+, \\
& x_2^- + \eta x_3^- + \eta^2 x_4^- + \cdots = -x_1^+, \\
& \dots
\end{aligned}
\right.
\end{equation}
This has the same property as before: if $x_0^+ = 0$, then a linear combination of the first two equations in \eqref{eq:match-3} shows that $x_0^- = 0$, and the remaining equations reproduce the original ones up to index shift. The analogue of \eqref{eq:k-equation} is
\begin{equation}
\left\{
\begin{aligned}
& x^-_0 + \eta x_1^- + \cdots + \eta^k x_k^- = 0, \\
& x_1^- + \eta x_2^- + \cdots + \eta^{k-1} x_k^- = -x_0^+, \\
& \dots \\
& 0 = -x_k^+, \\
& x^+_0 \neq 0, \\
& x^-_k \neq 0.
\end{aligned}
\right.
\end{equation}
(The remaining coordinates $x_j^+$, $j > k$, are also always zero.) One defines spaces $R_k^{p,\eta}$ as before, and those satisfy the analogue of \eqref{eq:boundary-theta}. If we set $\eta = 1$, then \eqref{eq:match-3} becomes equivalent to the $\theta = 1$ case of \eqref{eq:match-2}, so $R_k^{p,\eta=1} = R_k^{p,\theta=1}$. On the other hand, if we set $\eta = 0$, \eqref{eq:match-3} says that $x_0^- = 0$, and that the rest of $x^-$ agrees with $x^+$ up to index shift (and a $-1$ sign, which is of course irrelevant in projective space), so $R_k^{p,\eta = 0} = P_{k-1}^p$.

To define $R_k^p$, one takes the union of $R_k^{p,\theta}$ for $\theta$ lying on a path in the complex plane from $0$ to $1$; the same for $R_k^{p,\eta}$ and another path of the same kind; and the two pieces are then glued together by identifying $\theta = 1$ and $\eta = 1$ (to make the smooth structures match up, one chooses paths whose derivatives to all orders vanish as one approaches the endpoint $\theta = 1$ or $\eta = 1$).
\end{proof}

Taking circle bundles into account yields the following analogous statement:

\begin{lemma} \label{th:break-2}
There is a sequence of compact manifolds with corners $R_k^s$, $k>0$, satisfying
\begin{equation}
\partial R_k^s = Q_k^s \cup P_{k-1}^s \cup \Big( \bigcup_{k_1+k_2 = k} P_{k_1} \times R_{k_2}^s \Big) \cup \Big( \bigcup_{k_1+k_2=k} R_{k_1}^s \times P_{k_2} \Big),
\end{equation}
and these identifications are compatible with \eqref{eq:boundary-pk-s}, \eqref{eq:boundary-pk-s2}.
\end{lemma}

\begin{proof}
The proof is as before, only requiring minimal clarifications. To define the counterpart of $R_k^{s,\theta}$, one starts with a point $R_k^{p,\theta}$ and additionally chooses a preimage $w^- \in S^\infty$ of $[x^-] \in \bC P^\infty$ (this choice avoids the problem of $x^+$ not lying in $\bC P^\infty$ for general $\theta$). The same applies to $R_k^{s,\eta}$, and for $\eta = 0$ one can shift coordinates to obtain a preimage of $[x^+]$, which is used in the identification of that space with $P_{k-1}^s$. 
\end{proof}

From the proofs of these two Lemmas, one sees that the interpolating spaces are in fact fibrations over the parameter space (consisting of $\theta$ or $\eta$). Using a version of Ehresmann's theorem for manifolds with corners (whose proof follows the same strategy as for closed manifolds), one concludes that
\begin{align} 
\label{eq:trivial-topology}
& Q_k^p \iso P_{k-1}^p, && Q_k^s \iso P_{k-1}^s, \\
\label{eq:trivial-topology-2}
& R_k^p \iso [0,1] \times P_{k-1}^p, && R_k^s \iso [0,1] \times P_{k-1}^s.
\end{align}
One could choose such diffeomorphisms for all $k$, so that they are compatible with the recursive nature of the boundary strata (however, they are still non-canonical). There can be no analogue of \eqref{eq:trivial-topology} relating $Q_k^b$ and $P_{k-1}^b$, since the topologies differ even in lowest nontrivial dimension (in fact, the combinatorial structures of the boundary are also different, hence one can't even have a sensible cobordism type statement). Instead, we will determine the topology of $Q_k^b$ directly:

\begin{lemma} \label{th:s-cobordism-3}
For $k>1$, $Q_k^b$ is homeomorphic to $S^1 \times S^1 \times D^{2k-2}$.
\end{lemma}

\begin{proof}
By definition,
\begin{equation} \label{eq:int-qk-b}
Q_k^b \setminus \partial Q_k^b = \big\{ w = (w_0,\dots,w_k) \in S^{2k+1} \subset \bC^{k+1} \;:\;
w_0 \neq 0,\; w_k \neq 0, \; \textstyle \sum_j w_j < 0 \big\}.
\end{equation}
This is clearly a quotient of $\bC^{k+1} \setminus \{w_0 = 0 \text{ or } w_k = 0 \text{ or } \sum_j w_j = 0\}$ by the diagonal action of $\bC^*$, hence isomorphic to $\bC P^k \setminus \{\text{three hypersurfaces in general position}\} = \bC^* \times \bC^* \times \bC^{k-2}$. For $k>2$ one can use the $s$-cobordism theorem, as in Lemma \ref{th:s-cobordism}, to derive the result; we omit the case $k = 2$, which has to be settled by hand.
\end{proof}

\begin{lemma} \label{th:real-part}
The real part of $Q_2^b$ (the fixed part for the involution induced by complex conjugation on $\bC^\infty$) has four connected components. Their interiors are distinguished by having points $w$, as in \eqref{eq:int-qk-b} but with real coordinates, with signs 
\begin{equation} \label{eq:char-signs}
(\mathrm{sign}(w_0),\mathrm{sign}(w_2)) = (\pm,\pm).
\end{equation}
Each component is homeomorphic to a disc, and generates $H_2(Q_2^b,\partial Q_2^b) \iso \bZ$. The $(++)$ component is a triangle with
\begin{equation} \label{eq:triangle0}
\left\{
\begin{aligned}
& \text{one side lying in each of: $Q_2^s$, $P_1 \times Q_1^b$, $Q_1^b \times P_1$;} \\
& \text{one corner lying in each of: $P_1 \times Q_0^b \times P_1$, $P_1 \times Q_1^s$, $Q_1^s \times P_1$.}
\end{aligned}
\right.
\end{equation}
\end{lemma}

\begin{proof}
By the same argument as in Lemma \ref{th:s-cobordism-3}, the interior of the real part, denoted by $(Q_2^b \setminus \partial Q_2^b)^{\bR}$, is diffeomorphic to $\bR P^2 \setminus \{\text{three real hypersurfaces in general position}\}$. This also immediately yields \eqref{eq:char-signs}. Inspection of the proof of Lemma \ref{th:s-cobordism-3} shows that $H_2(Q_2^b) \iso \bZ$ is generated by the homology class of the torus 
\begin{equation}
\{ |w_0| = \epsilon, |w_2| = \epsilon \} \subset Q_2^b \setminus \partial Q_2^b \quad \text{for sufficiently small $\epsilon>0$.}
\end{equation}
This intersects each component of $(Q_2^b)^{\bR}$ transversally in one point, which implies the desired homological statement. 
The boundary of the $(++)$ component of $(Q_2^b)^{\bR}$ contains exactly one interval which belongs to $(Q_2^s)^{\bR}$. Because of the way in which the boundary components intersect, it then also contains exactly one interval each in $(Q_1^b \times P_1)^{\bR}$ and $(P_1 \times Q_1^b)^{\bR}$. It contains no points belonging to the other codimension one boundary faces $Q_0^b \times P_2$ and $P_2 \times Q_0^b$ (because on those faces, $w_2=-1$ or $w_0=-1$).
\end{proof}
%

Our main application of moduli spaces with evaluation constraints is to construct certain other spaces, which bound our previous $P_{k-1}^s$. Namely, take $Q_k^b$ and $R_k^s$, and glue them together along their common boundary face $Q_k^s$. This yields a sequence of manifolds with corners, denoted by $R_k^b$, satisfying 
\begin{equation} \label{eq:break-3}
\partial R_k^b = P_{k-1}^s \cup \Big( \bigcup_{k_1+k_2 = k} P_{k_1} \times R_{k_2}^b \Big) \cup \Big(
\bigcup_{k_1+k_2 = k} R_{k_1}^b \times P_{k_2} \Big).
\end{equation}
In the lowest-dimensional cases,
\begin{equation}
\begin{aligned}
& R_0^b = Q_0^b = \mathit{point}, \\
& R_1^b = Q_1^b \cup_{Q_1^s} R_k^s = \text{\it pair-of-pants} \cup_{S^1} \mathit{annulus} \iso
\text{\it pair-of-pants}.
\end{aligned}
\end{equation}
In fact, from \eqref{eq:trivial-topology}, \eqref{eq:trivial-topology-2} it follows that $R_k^b \iso Q_k^b$ for all $k$. 

\begin{lemma} \label{th:alphabeta-p}
Suppose that $(\alpha_k)$ and $(\alpha_k^s,\beta_k^s)$ have been chosen, as in Lemmas \ref{th:alpha} and \ref{th:alphabeta-s}. Then, there are smooth maps 
\begin{equation}
\alpha_k^b,\, \beta_k^b: R_k^b \longrightarrow S^1,
\end{equation}
whose restriction to \eqref{eq:break-3} is given by
\begin{equation} \label{eq:middle-alphabeta-p}
\left\{
\begin{aligned}
& \alpha_k^b \,|\, P_{k-1}^s = \alpha_{k-1}^s, \\
&\beta_k^b \,|\, P_{k-1}^s = \beta_{k-1}^s
\end{aligned}
\right.
\end{equation}
as well as
\begin{equation} \label{eq:left-alphabeta-p}
\left\{
\begin{aligned} 
& \alpha_k^b \,|\, (P_{k_1} \times R_{k_2}^b) = \alpha_{k_1} + \alpha_{k_2}^b, \\
& \beta_k^b \,|\, (P_{k_1} \times R_{k_2}^b) = -\alpha_{k_1} + \beta_{k_2}^b,
\end{aligned}
\right.
\end{equation}
and
\begin{equation}
\left\{
\begin{aligned} \label{eq:right-alphabeta-p}
& \alpha_k^b \,|\, (R_{k_1}^b \times P_{k_2}) = \alpha_{k_1}^b + \alpha_{k_2}, \\
& \beta_k^b \,|\, (R_{k_1} ^b \times P_{k_2}) = \beta_{k_1}^b.
\end{aligned}
\right.
\end{equation}
\end{lemma}

\begin{proof}
In view of the more complicated construction of the $R_k^b$, we will use an abstract topological argument, along the line of the second proofs of Lemmas \ref{th:alpha} and \ref{th:alphabeta-s}.

Let's start with $k = 1$. Equip the interior of $Q_1^b$, thought of as in \eqref{eq:3-punctured}, with its complex orientation. The boundary circle corresponding to the puncture at $-1$ is identified with $Q_1^s$, but its boundary orientation is the opposite of the natural orientation of $Q_1^s$ (by which we mean, the orientation $Q_1^s \iso S^1$ inherits from being a fibre of $S^\infty \rightarrow \bC P^\infty$). By definition, $R_1^b$ is obtained by attaching $R_1^s$ to that boundary circle of $R_1^b$. Now, $R_1^s$ is a circle bundle over the interval $R_1^p$. Hence, the previous observation carries over: the boundary orientation of $P_0^s \subset \partial R_1^b$ is opposite to its natural orientation. The same orientation behaviour appears at the boundary circle coming the puncture at $0$ in \eqref{eq:3-punctured}, which corresponds to $P_1 \times Q_0^b = P_1 \times R_0^b \subset \partial R_1^b$. In contrast, on the boundary circle coming from the puncture at $\infty$, which corresponds to $Q_0^b \times P_1 = R_0^b \times P_1 \subset \partial R_1^b$, the two orientations agree. With that taken into account, the condition \eqref{eq:middle-alphabeta-p} says that on $P_0^s \subset \partial R_1^b$ (equipped with its boundary orientation), $\alpha_1^p$ has degree $0$, while $\beta_1^p$ has degree $-1$. From \eqref{eq:left-alphabeta-p} one gets that on $P_1 \times R_0^b$ (again, equipped with its boundary orientation), $\alpha_1^p$ has degree $-1$, and $\beta_1^p$ has degree $1$. Similarly by \eqref{eq:right-alphabeta-p}, on $R_0^b \times P_1$, $\alpha_1^p$ has degree $1$, while $\beta_1^p$ has degree 0. Given functions on the boundary with these properties, one can therefore extend them to all of $R_1^b$. 

By the recursive conditions, these choices determine the values of our functions on $\partial R_2^b \iso (S^1)^3$. A generator of $H_2(R_2^b,\partial R_2^b)$ can be constructed as in Lemma \ref{th:real-part}. Namely, take the $(++)$ component of $(Q_2^b)^{\bR}$, and then attach to part of its boundary the corresponding component of $(R_2^s)^{\bR}$ (this makes sense provided that $R_2^s$ carries an appropriate real involution; this can be ensured by taking the parameters $\theta$ and $\eta$ from the proof of Lemma \ref{th:break-1}, or rather their counterparts in Lemma \ref{th:break-2}, to lie on the real axis). The outcome is a triangle in $R_2^b$ which, because of \eqref{eq:triangle0}, has:
\begin{equation} \label{eq:triangle1}
\left\{
\begin{aligned}
& \text{one side each lying in $P_1^s$, $P_1 \times R_1^b$ and $R_1^b \times P_1$;} \\
& \text{one corner each lying in $P_1 \times R_0^b \times P_1$, $P_1 \times P_0^s$, $P_0^s \times P_1$.}
\end{aligned}
\right.
\end{equation}
If we take the two last-mentioned sides in \eqref{eq:triangle1} and project them to the $R_1^b$ factors, we get two paths, one going from $R_0^b \times P_1$ to $P_0^s$, and the other from $P_1 \times R_0^b$ to $P_0^s$. Because all three boundary circles of $R_1^b$ appear in this way, our two paths generate $H_1(R_1^b,\partial R_1^b) \iso \bZ^2$. Now let's go back to the previous step: when defining $\alpha_1^b$ and $\beta_1^b$, we were free to add arbitrary functions $R_1^b \rightarrow S^1$ which vanish on the boundary. By modifying our choice in such a way, one can always achieve that $\alpha_2^b$ and $\beta_2^b$ have degree zero along the boundary of our triangle. This is a necessary and sufficient condition for extendibility to $R_2^b$.

Finally, suppose that, for some $l \geq 3$, we have defined $(\alpha_k^b, \beta_k^b)$ for all $k<l$, with the desired properties. This determines the values of $(\alpha_l^b,\beta_l^b)$ on $\partial R_l^b$. By construction, $R_l^b \iso Q_l^b$. From this and Lemma \ref{th:s-cobordism-3}, one sees that $(R_l^b,\partial R_l^b)$ is $2$-connected; hence, the extensions of our functions over $R_l^b$ is always possible.
\end{proof}

We will also we need to consider orientation issues for the higher-dimensional moduli spaces. Equip $Q_k^p$ with their complex orientations. For the circle bundles $Q_k^s \rightarrow Q_k^p$, we then choose orientations in the same way as in \eqref{eq:circle-orientation-sequence}. Choose orientations of $Q_k^b$ which are compatible with those of $Q_k^s \subset \partial P_k^s$, and extend them to orientations of $R_k^b$. Then, we have the following statement, whose proof we omit:

\begin{lemma} \label{th:r-orientation}
The orientations of $R_k^b$ are compatible with the orientations of the boundary faces $P_{k-1}^s$. The same holds for boundary faces $P_{k_1} \times R_{k_2}^b$, while for those of the form $R_{k_1}^b \times P_{k_2}$ the orientations are opposite.
\end{lemma}

Finally, we need to discuss the analogues for our moduli spaces of the tautological families \eqref{eq:tautological}. Clearly, the spaces $Q_k^p$, $Q_k^s$ and $Q_k^b$ each carry such a family, with a distinguished section as in \eqref{eq:canonical-section}, and with the usual kind of compactification. Slightly less obviously, the same holds for the spaces $R_k^p$, $R_k^s$ and $R_k^b$, except that the total spaces of the tautological families no longer come with maps to $\bC P^\infty$. Take for instance one of the spaces $R_k^\theta$ appearing in the proof of Lemma \ref{th:break-1}, and a point in its interior, represented by a pair of half-flow-lines $(v^-,v^+)$. Over this point, the fibre of the tautological family can be a point on either half-flow line, which means either $s^- \in (-\infty,0]$ or $s^+ \in [0,\infty)$, with the convention that we identify $s^- = 0$ with $s^+ = 0$; but the evaluation maps $v^-(s^-)$ and $v^+(s^+)$ fail to respect that identification, for $\theta \neq 0$. The tautological family over $R_k^p$ restricts to that for $Q_k^p$ on the appropriate boundary face. One then defines the (most complicated) family over $R_k^b$ by gluing together those on $Q_k^b$ and $R_k^s$, just as in the definition of the space $R_k^b$ itself.

\section{Floer cohomology background\label{sec:floer}}

This section reviews Hamiltonian Floer cohomology and some of its properties, selected with a view to their usefulness later on. The technical choices made in presenting the construction largely follow classical models, specifically \cite{salamon-zehnder92, floer-hofer-salamon94, hofer-salamon95, piunikhin-salamon-schwarz94}.

\subsection{Geometric setup}
Let $(M,\omega)$ be a $2n$-dimensional compact symplectic manifold with boundary. We assume that \eqref{eq:cy} holds. We also fix oriented codimension two submanifolds $\Omega_1,\dots,\Omega_j \subset M$ (which are allowed to have boundary on $\partial M$), and multiplicities $m_1,\dots,m_j \in A$, such that the cycle $\Omega = m_1 \Omega_1 + \cdots + m_j \Omega_j$ satisfies 
\begin{equation} \label{eq:omega-cycle}
[\Omega] = m_1 [\Omega_1] + \cdots + m_j [\Omega_j] \in H^2(M;A) \longmapsto 
[\omega] \in H^2(M;\bR).
\end{equation}
On the complement of $\Omega$, there is a one-form $\theta$ satisfying $d\theta = \omega$, and with the following property. Whenever $S$ is a compact oriented surface with boundary, and $u: S \rightarrow M$ a map such that $u(\partial S) \cap \Omega = \emptyset$, then
\begin{equation} \label{eq:energy-omega}
\int_S u^*\omega = u \cdot \Omega + \int_{\partial S} u^*\theta.
\end{equation}
That concludes the topological part of our setup, and we now turn to Hamiltonian dynamics and holomorphic curve theory. We assume that $M$ comes with a function $\scrH$ such that:
\begin{equation} \label{eq:h}
\parbox{35em}{$\scrH$ is locally constant on $\partial M$, with the gradient pointing outwards; in particular, there are no critical points on $\partial M$.}
\end{equation}
Let $\scrX$ be the Hamiltonian vector field of $\scrH$. We also assume that $M$ comes with a compatible almost complex structure $\scrJ$, such that the following holds:
\begin{equation} \label{eq:convexity}
\parbox{35em}{
$\partial M$ is weakly Levi convex with respect to $\scrJ$. This means that $-d(d\scrH \circ \scrJ)$ is nonnegative on each $\scrJ$-complex line in $T(\partial M)$. Additionally, we assume that $L_{\scrX} (d\scrH \circ \scrJ)$ vanishes along the boundary.
}
\end{equation}
The use of this kind of convexity condition in pseudoholomorphic curve theory is classical, but for convenience, we will describe its implication in a basic form:

\begin{lemma} \label{th:levi}
Let $S$ be a connected Riemann surface, with complex structure $j$, equipped with a one-form $\beta \in \Omega^1(S,\bR)$ such that $d\beta \leq 0$. Consider maps $u: S \rightarrow M$ which satisfy
\begin{equation} \label{eq:dbar}
(Du - \scrX \otimes \beta)^{0,1} = \half (Du + \scrJ \circ Du \circ j - \scrX \otimes \beta -
\scrJ\scrX \otimes \beta \circ j) = 0.
\end{equation}
If such a map meets $\partial M$, it must be entirely contained in it.
\end{lemma}

\begin{proof}
It is well-known that one can rewrite \eqref{eq:dbar} as the property of $\tilde{u}(z) = (z,u(z))$ to be a pseudo-holomorphic map into $\tilde{M} = S \times M$, with respect to the almost complex structure $\tilde{\scrJ}$ defined by
\begin{equation}
\left\{
\begin{aligned}
& \tilde{\scrJ} \xi = \scrJ \xi && \text{for $\xi \in TM$,} \\
& \tilde{\scrJ} (\eta + \scrX \beta(\eta)) = j\eta + \scrX \beta(j\eta) && \text{for $\xi \in TS$.}
\end{aligned}
\right.
\end{equation}
Let $\tilde{\scrH}$ be the pullback of $\scrH$ to $\tilde{M}$. The Levi form at a boundary point of $\tilde{M}$ is
\begin{equation} \label{eq:twisted-levi}
-d(d\tilde{\scrH} \circ \tilde{\scrJ})(\eta + \scrX \beta(\eta) + \xi, j\eta + \scrX \beta(j\eta) + \scrJ \xi) =
-\omega(\scrX, \scrJ \scrX) d\beta(\eta, j\eta) - d(d\scrH \circ \scrJ)(\xi, \scrJ \xi) 
\end{equation}
for $\eta \in TS$, $\xi \in T(\partial M) \cap \scrJ T(\partial M)$. By assumption, \eqref{eq:twisted-levi} is nonnegative. One now applies \cite[Corollary 4.7]{diederich-sukhov08} to show that if $\tilde{u}$ meets $\partial \tilde{M}$, it must be entirely contained in it.
\end{proof}

\begin{application} \label{th:contact}
The most commonly studied situation where \eqref{eq:h} and \eqref{eq:convexity} hold (and the one we adopted when stating our results in Section \ref{sec:results}) is that of a symplectic manifold with convex contact type boundary, where one takes the Hamiltonian flow to be an extension of the Reeb flow on the boundary (see e.g.\ \cite{viterbo}). Let $M$ be such a manifold, and $\scrZ$ a Liouville vector field, defined near $\partial M$. We then define $\scrH$ near $\partial M$ by asking that
\begin{equation} \label{eq:h-lambda}
\begin{aligned} 
& \scrH|\partial M = 1, \\ 
& \scrZ.\scrH = \scrH;
\end{aligned}
\end{equation}
the associated $\scrX$ restricts to the Reeb vector field on $\partial M$. One chooses $\scrJ$ so that 
\begin{equation}
-d \scrH \circ \scrJ = \omega(\scrZ,\cdot)
\end{equation}
is the primitive of the symplectic form near $\partial M$.
\end{application}

We will only really use the Taylor expansion to second order of $\scrH$ along the boundary, and the corresponding first order expansion of $\scrJ$, which are enough in order for \eqref{eq:h} and \eqref{eq:convexity} to make sense. Those data are considered to be part of the structure of $M$, and will be kept fixed. From now on, when we use Hamiltonian functions $H$ on $M$, these are always assumed to agree with some multiple $\epsilon\scrH$ to second order along $\partial M$. Here, $\epsilon>0$ is such that:
\begin{equation} \label{eq:epsilon}
\parbox{35em}{$\scrX|\partial M$ has no $\epsilon$-periodic orbits.}
\end{equation}
Similarly, all almost complex structures $J$ will be assumed to agree with $\scrJ$ to first order along $\partial M$. 

\subsection{Floer cohomology\label{subsec:floer}}
Choose a time-dependent Hamiltonian $H = (H_t)_{t \in S^1}$ (in the class defined above, for some $\epsilon$), and let $X = (X_t)$ be the associated vector field. Consider $1$-periodic orbits, which means solutions
\begin{equation} \label{eq:orbit}
\left\{
\begin{aligned}
& x: S^1 \longrightarrow M, \\
& dx/dt = X_t.
\end{aligned}
\right.
\end{equation}
All such orbits lie in the interior of $M$, by construction, and we additionally assume that they should be nondegenerate and disjoint from $\Omega$. From now on, we will only use $1$-periodic orbits which are nullhomologous, meaning that
\begin{equation} \label{eq:h1-null}
[x] = 0 \in H_1(M)
\end{equation}
(for a version involving non-nullhomologous orbits as well, see Remark \ref{th:non-contractible}). We also choose a time-dependent almost complex structure $J = (J_t)$. The construction of Floer cohomology is based on solutions of 
\begin{equation} \label{eq:floer}
\left\{
\begin{aligned}
& u: \bR \times S^1 \longrightarrow M, \\
& \partial_s u + J_t(\partial_t u - X_t) = 0, \\
& \textstyle \lim_{s \rightarrow \pm \infty} u(s, t) = x_\pm(t).
\end{aligned}
\right.
\end{equation}
Here, $x_{\pm}$ are orbits \eqref{eq:orbit} satisfying \eqref{eq:h1-null}. (There is a budding notational clash, between pseudo-holomorphic maps $u$ on one hand, and the formal variable $u$ in the equivariant theory on the other hand. Since the objects involved are so different, chances of confusion are hopefully minimal.) The operations on Floer groups that will appear use more general ``continuation map equations'', of the overall form
\begin{equation} \label{eq:floer-2}
\left\{
\begin{aligned}
& u: \bR \times S^1 \longrightarrow M, \\
& \partial_s u + J_{s,t}^\ast(\partial_t u - X_{s,t}^\ast) = 0, \\
& \textstyle \lim_{s \rightarrow -\infty} u(s,t) = x_-(t), \\
& \textstyle \lim_{s \rightarrow +\infty} u(s,t) = x_+(t+\tau). 
\end{aligned}
\right.
\end{equation}
Here, we have chosen $\tau \in S^1$, and a family of functions and almost complex structures $(H_{s,t}^\ast,J_{s,t}^\ast)$, with associated vector field $X_{s,t}^\ast$, such that:
\begin{equation} \label{eq:tau-hj}
(H_{s,t}^\ast,J_{s,t}^\ast) = \begin{cases} 
(H_t,J_t) & s \ll 0, \\
(H_{t+\tau},J_{t+\tau}) & s \gg 0.
\end{cases} 
\end{equation}
(Later, when many different choices of $(H_{s,t}^\ast,J_{s,t}^\ast)$ will occur, the superscript $\ast$ will be replaced by the name of the Floer-theoretic operation under construction.) The basic analytic aspects of  \eqref{eq:floer} and \eqref{eq:floer-2} are familiar:
\begin{itemize}
\itemsep.5em
\item No solution can reach $\partial M$. To see that, one follows the argument from Lemma \ref{th:levi}, for $S = \bR \times S^1$ and $\beta = \epsilon \mathit{dt}$. The almost complex structure on $\tilde{M}$ constructed from $(H_{s,t}^*, J_{s,t}^*)$ still satisfies \eqref{eq:twisted-levi}. Because the limits of $u$ lie in the interior, it is impossible for that map to be entirely contained in $\partial M$, and this concludes the argument. The same property holds for the nodal pseudo-holomorphic curves produced from sequences of solutions by sphere bubbling.

\item Let $i(x) \in \bZ$ be the Conley-Zehnder index of a $1$-periodic orbit, which is well-defined thanks to \eqref{eq:cy}. The linearization of our equation is a Fredholm operator $D_u$ with
\begin{equation} \label{eq:index}
\mathrm{index}(D_u) = i(x_-) - i(x_+).
\end{equation}

\item \parskip1em \parindent0em
Transversality issues can be dealt with by varying the auxiliary data, as in \cite{floer-hofer-salamon94, hofer-salamon95} (and the same applies to ``transversality of evaluation''). The fact that those data have to be kept fixed along $\partial M$ does not affect our argument, since solutions remain in the interior.

\item
One defines the action of a $1$-periodic orbit to be
\begin{equation}
A_H(x) = \int_{S^1} -x^*\theta + H_t(x(t)) \, \mathit{dt}.
\end{equation}
Then, the energy of any solution $u$ can be written as
\begin{equation} \label{eq:energy}
E(u) = \int_{\bR \times S^1} |\partial_s u|^2 = 
A_H(x_-) - A_H(x_+) + u \cdot \Omega + \int_{\bR \times S^1} (\partial_s H_{s,t}^*)(u(s,t)).
\end{equation}
The last term is bounded independently of $u$, because of \eqref{eq:tau-hj}. Hence, an upper bound on the intersection number $u \cdot \Omega \in A$ yields a bound on the energy. 
\end{itemize}

The Floer cochain complex is
\begin{equation}
\mathit{CF}^*(H) = \bigoplus_x \Lambda_x.
\end{equation}
Here, the sum is over all $1$-periodic orbits $x$, each of which contributes a one-dimensional summand $\Lambda_x$ (identified with $\Lambda$ in a way that's canonical up to a sign, and placed in degree $i(x)$; we will usually write $\pm x$ for the preferred generators of that summand). The differential
\begin{equation} \label{eq:d}
\begin{aligned} 
& d: \mathit{CF}^*(H) \longrightarrow \mathit{CF}^{*+1}(H), \\
& dx_+ = \sum_u \pm q^{u \cdot \Omega}\, x_-,
\end{aligned}
\end{equation}
is obtained by counting solutions of \eqref{eq:floer}. More precisely, one counts non-stationary solutions which are isolated up to $s$-translation, with signs (given more intrinsically by an isomorphism $\Lambda_{x_+} \rightarrow \Lambda_{x_-}$ for each $u$). Up to canonical isomorphism, the resulting Floer cohomology depends only on $\epsilon$; we denote it by $\mathit{HF}^*(M, \epsilon)$. 

\begin{example}
As a very simple instance of the well-definedness property, let's see how, $(H,J)$ being kept fixed, Floer cohomology is independent of $\Omega$. Suppose that we have two choices $\Omega_{\pm}$, with associated $\theta_{\pm}$. Take a cycle $\tilde{\Omega}$ in $\bR \times M$ with $A$-coefficients, which interpolates between the two, meaning that it equals $\bR^{\pm} \times \Omega_{\pm}$ at infinity. Then, the associated map between Floer cochain complexes simply rescales each generator by a suitable power of $q$:
\begin{equation} \label{eq:rescale}
r(x) = q^{\tilde{A}(x)} x,
\end{equation}
where
\begin{equation} \label{eq:tilde-a}
\tilde{A}(x) = (\bR \times x) \cdot \tilde{\Omega} = 
-\int_{S^1} x^*\theta_+ + \int_{S^1} x^*\theta_-.
\end{equation}
The two expressions for $\tilde{A}(x)$ show that it lies in the subgroup $A$, and also that it is independent of the choice of $\tilde{\Omega}$; the equivalence of those expressions is shown by capping off $x_{\pm}$ with surfaces in $M$ (hence uses the fact that $x_{\pm}$ is nullhomologous).
\end{example}

\subsection{Operations}
The general structure of operations on Floer cohomology is roughly as follows. Suppose that we have an equation \eqref{eq:floer-2}, where the data $(H^\ast_{s,t},J^\ast_{s,t})$ can depend on additional parameters. The simplest case is when the parameter space is a compact oriented manifold with boundary, denoted by $P$. Then (assuming suitably generic choices to ensure transversality), counting isolated points in the parametrized moduli space of solutions of \eqref{eq:floer-2}, in the same way as in \eqref{eq:d}, yields a map
\begin{equation}
\phi_P : \mathit{CF}^*(H) \longrightarrow \mathit{CF}^{*-\mathrm{dim}(P)}(H),
\end{equation}
which is related to its counterpart for the restriction of the parameters to $\partial P$ by 
\begin{equation} \label{eq:sign-1}
(-1)^{\mathrm{dim}(P)} d\phi_P - \phi_P d + \phi_{\partial P} = 0.
\end{equation}
If $P$ is closed, $\phi_P$ is a chain map of degree $-\mathrm{dim}(P)$. A standard generalization is where 
\begin{equation} \label{eq:boundary-product}
\partial P = P_1 \times P_2
\end{equation}
is a product, and the family of equations \eqref{eq:floer-2} does not smoothly extend to the boundary, but instead asymptotically decouples into two equations parametrized by $P_1$ and $P_2$, which are limits over parts of the cylinder that are separated by an increasingly long neck. In that case, the modified formula for the boundary contribution in \eqref{eq:sign-1} is
\begin{equation} \label{eq:added-koszul}
\phi_{\partial P} = (-1)^{\mathrm{dim}(P_1) \mathrm{dim}(P_2)} \phi_{P_1} \phi_{P_2}.
\end{equation}
There are further generalizations of those basic setups, involving parameter spaces that are manifolds with corners. These are all routinely used in Floer theory, and we will not spell out the details; the very short discussion here was intended merely as an indication of our notation and sign conventions.

\begin{remark}
The Koszul sign in \eqref{eq:added-koszul} may deserve some explanation. Points in a para\-me\-trized moduli space are pairs $(r,u)$ consisting of some $r \in P$ and a map $u$ satisfying the appropriate $r$-dependent equation \eqref{eq:floer-2}. Linearizing the equation (with variable $r$) yields an operator which is an extension of the ordinary linearized operator $D_u$, with the domain enlarged by $T_r P$. The top exterior power of the tangent space of the parametrized moduli space is the determinant line of this extended operator, which can be identified with
\begin{equation} \label{eq:det}
\lambda^{\mathit{top}}(T_r P) \otimes \mathit{det}(D_u).
\end{equation}
In the limit where $r$ degenerates to a point $(r_1,r_2) \in P_1 \times P_2$, and $u$ converges to a limit consisting of pieces $(u_1,u_2)$, the corresponding expression along the boundary would be
\begin{equation} \label{eq:det-2}
\lambda^{\mathit{top}}(T_{r_1} P_1) \otimes \mathit{det}(D_{u_1}) \otimes \lambda^{\mathit{top}}(T_{r_2} P_2) \otimes \mathit{det}(D_{u_2}).
\end{equation}
To compare \eqref{eq:det} and \eqref{eq:det-2}, one uses the isomorphism $\lambda^{\mathit{top}}(TP)|\partial P \iso \lambda^{\mathit{top}}(TP_1) \otimes \lambda^{\mathit{top}}(TP_2)$ induced by \eqref{eq:boundary-product}, as well as the gluing formula for determinant lines, $\mathit{det}(D_u) \iso \mathit{det}(D_1) \otimes \mathit{det}(D_2)$. When applying those two results to \eqref{eq:det-2}, one exchanges the middle two factors in the tensor product, and that comes with a Koszul sign $(-1)^{\mathrm{index}(D_{u_1}) \mathrm{dim}(P_2)}$, by the construction of determinant line bundles (see \cite{zinger16} for a comprehensive exposition). But since the relevant argument considers isolated solutions $(r_1,u_1)$ and $(r_2,u_2)$, the index of $D_{u_1}$ is minus the dimension of $P_1$.
\end{remark}

As a warmup for later considerations, we want to discuss certain specific operations. The simplest of these is the quantum cap product with the class $q^{-1}[\Omega] \in H^2(M;\Lambda)$. To define the underlying chain map
\begin{equation} \label{eq:iota}
\iota: \mathit{CF}^*(H) \longrightarrow \mathit{CF}^{*+2}(H),
\end{equation}
choose some $(H^\iota,J^\iota) = (H^\iota_{s,t}, J^\iota_{s,t})$ which satisfies \eqref{eq:tau-hj} with $\tau = 0$, meaning that it reduces to $(H_t,J_t)$ for $|s| \gg 0$. Then, consider solutions of the associated equation \eqref{eq:floer-2} which satisfy the incidence condition
\begin{equation} \label{eq:incidence}
u(0,0) \in q^{-1}\Omega.
\end{equation}
(There is nothing special about $(0,0)$: any other point on the cylinder could be used instead. Similarly, we could have used any fixed value of $\tau$.) The notation \eqref{eq:incidence} is shorthand for the following. For each component $\Omega_j$, we count solutions such that $u(0,0) \in \Omega_j$ as usual with $\pm q^{u \cdot \Omega}$, and then take the sum of those contributions with multiplicities $q^{-1}m_j$ taken from \eqref{eq:omega-cycle}. Obviously, one has to assume that the space of solutions satisfies suitable transverse intersections conditions with the $\Omega_j$, but that is easy to achieve, given the freedom to choose $(H_{s,t}^\iota, J_{s,t}^\iota)$. Instead, one could also opt for a more restricted choice, which is to just use the given $(H^\iota_{s,t},J^\iota_{s,t}) = (H_t,J_t)$. This would require an additional transversality argument for the original $(H_t,J_t)$, which is again within the scope of standard methods.

There is a similar operation where one allows the evaluation point to move,
\begin{equation} \label{eq:lambda}
\lambda: \mathit{CF}^*(H) \longrightarrow \mathit{CF}^{*+1}(H).
\end{equation}
For that, one introduces a parameter $r \in S^1$, and replaces \eqref{eq:incidence} with
\begin{equation} \label{eq:incidence-2}
u(0,-r) \in q^{-1}\Omega.
\end{equation}
As before, one implements this by choosing a family $(H^\lambda,J^\lambda) = (H^\lambda_{r,s,t}, J^\lambda_{r,s,t})$ which, for each value of $r$, satisfies \eqref{eq:tau-hj} with $\tau = 0$. Alternatively, the special choice $(H^\lambda_{r,s,t},J^\lambda_{r,s,t}) = (H_t,J_t)$ also still works, assuming suitable transversality properties. The advantage of adopting this special choice (which will be crucial later on) is that then, \eqref{eq:iota} can be viewed as a sum over the same solutions as in the Floer differential, but with modified multiplicities:
\begin{equation} \label{eq:concrete-lambda}
\lambda(x_+) = \sum_u \pm q^{u \cdot \Omega-1}\, (u \cdot \Omega) x_-.
\end{equation}
The idea is that, if $u(s,t) \in \Omega$, one can translate $u$ in $s$-direction so that \eqref{eq:incidence-2} holds, with $r =-t$. The sign in \eqref{eq:incidence-2} may seem puzzling in view of \eqref{eq:concrete-lambda}; we refer to \cite[Section 8a]{seidel16} for a detailed explanation.

The final operation we want to consider is the BV operator
\begin{equation} \label{eq:bv}
\Delta: \mathit{CF}^*(H) \longrightarrow \mathit{CF}^{*-1}(H).
\end{equation}
Again, this is based on a moduli problem with one parameter $r \in S^1$, but where that parameter now affects the rotation of the end $s \rightarrow \infty$. Concretely, this means that one chooses $(H^\Delta,J^\Delta) = (H^\Delta_{r,s,t}, J^\Delta_{r,s,t})$ satisfying \eqref{eq:tau-hj} for 
\begin{equation} \label{eq:tau-is-r}
\tau = r, 
\end{equation}
and uses the resulting parametrized space of solutions of \eqref{eq:floer-2} (more generally, one could let $\tau = \tau(r)$ be any degree $1$ function $S^1 \rightarrow S^1$). This time, there is no option to use the original $(H,J)$, because they are not time-independent (there are special cases where this is possible, leading to vanishing of $\Delta$; see Section \ref{subsec:morse-floer} below).

\begin{lemma} \label{th:iota-lambda}
There is a chain homotopy
\begin{equation} \label{eq:iota-lambda}
\lambda \htp \Delta \iota - \iota \Delta.
\end{equation}
\end{lemma}

\begin{proof}
We begin by rewriting the two terms on the right hand side in a more compact way, up to chain homotopy. Namely, consider a setup which still has a parameter $r \in S^1$, with $\tau = r$, and additionally the incidence condition \eqref{eq:incidence-2}. This gives rise to an operation
\begin{equation}
\label{eq:lambda-plus}
\lambda_+ \htp \Delta \iota.
\end{equation}
To get the homotopy in \eqref{eq:lambda-plus}, one uses a neck-stretching argument in which our family degenerates to that for $\Delta$, glued together with the surface underlying $\iota$ (one can also think of this argument as moving the marked point towards $s \rightarrow +\infty$). This works because (to put it in the simplest terms) after a coordinate change $\tilde{u}(s,t) = u(s,t-r)$, part of the original conditions looks like this:
\begin{equation}
\begin{aligned}
& \tilde{u}(0,0) \in q^{-1}\Omega, \\
& \textstyle \lim_{s \rightarrow +\infty} \tilde{u}(s,t) = x_+(t).
\end{aligned}
\end{equation}
On the other hand, one can consider another parametrized moduli problem, where still $\tau = r$, but the incidence condition is the $r$-independent one \eqref{eq:incidence}. By a similar argument, this gives rise to an operation
\begin{equation}
\label{eq:lambda-minus}
\lambda_- \htp \iota \Delta.
\end{equation}
There is another family of equations \eqref{eq:floer-2} parametrized by the compact pair-of-pants, whose restrictions to the three boundary circles are: the family underlying $\lambda_+$; that underlying $\lambda_-$, with the orientation of the circle reversed; and the family underlying $\lambda$, again with reversed orientation. From that, one gets a homotopy
\begin{equation} \label{eq:lambda-plus-minus}
\lambda_+ - \lambda_- \htp \lambda.
\end{equation}
By combining \eqref{eq:lambda-plus}, \eqref{eq:lambda-minus} and \eqref{eq:lambda-plus-minus}, one obtains \eqref{eq:iota-lambda}. We have divided the construction of \eqref{eq:iota-lambda} into three parts for ease of exposition. However, one can also implement it as a single homotopy given by a combined parametrized moduli problem, where the parameter space is a modified pair-of-pants (with one boundary circle and two ends; equivalently, a closed disc with two interior points removed).
\end{proof}

\begin{lemma}
There is a nullhomotopy
\begin{equation} \label{eq:delta-lambda}
\lambda \Delta + \Delta \lambda \htp 0.
\end{equation}
\end{lemma}

\begin{proof}
Each of the two terms in the equation is chain homotopic to what one would get from a moduli problem with parameters in $S^1 \times S^1$. For $\Delta \lambda$, we denote the parameters by $(r_1^+,r_2^+)$, and the conditions are 
\begin{equation} \label{eq:plus-conditions}
\tau = r_1^+, \quad u(0,-r_1^+ - r_2^+) \in q^{-1}\Omega.
\end{equation}
For $\lambda \Delta$, we denote the parameters by $(r_1^-,r_2^-)$ and the counterpart of \eqref{eq:plus-conditions} is
\begin{equation} \label{eq:minus-conditions}
\tau = r_2^-, \quad u(0,-r_1^-) \in q^{-1}\Omega.
\end{equation}
(More precisely, these families give operations homotopic to $-\Delta\lambda$ and $-\lambda\Delta$, because of the sign in \eqref{eq:added-koszul}, but that ultimately makes no difference to our argument.) The conditions \eqref{eq:plus-conditions} and \eqref{eq:minus-conditions} are related by an orientation-reversing parameter change
\begin{equation} \label{eq:two-dim-parameter-space}
(r_1^+, r_2^+) = (r_2^-, r_1^- - r_2^-).
\end{equation}
One can therefore combine the two chain homotopies to get \eqref{eq:delta-lambda}.
As before, one could also encode the entire argument in a single parametrized moduli problem, with parameter space $\bR \times S^1 \times S^1$.
\end{proof}

\begin{lemma} \label{th:delta-squared}
There is a nullhomotopy
\begin{equation} \label{eq:delta-squared}
\Delta^2 \htp 0.
\end{equation}
\end{lemma}

\begin{proof}
This is the most familiar among our relations. As before, $\Delta^2$ is chain homotopic to what one gets from a moduli problem with parameters $(r_1,r_2) \in S^1 \times S^1$, and which has
\begin{equation}
\tau = r_1 + r_2.
\end{equation}
Since only $r_1+r_2$ appears, one can extend the relevant family over the solid torus, and that yields the nullhomotopy.
\end{proof}

The preceding three Lemmas are not independent; in view of \eqref{eq:delta-squared}, \eqref{eq:iota-lambda} clearly implies \eqref{eq:delta-lambda}. Nevertheless, we have explained them separately, since each argument forms the toy model for one of the constructions that follow.

\subsection{Relation with Morse theory\label{subsec:morse-floer}}
Our next task is to review the isomorphism between ordinary cohomology and Floer cohomology, which holds when the Hamiltonian is sufficiently small. Ordinary cohomology will be realized through Morse theory. Let $f$ be a Morse function which is locally constant on $\partial M$, with the gradient pointing outwards. After choosing a metric $g$ which makes $\nabla f$ Morse-Smale, one can associate to it the Morse complex $\mathit{CM}^*(f)$ (with $\Lambda$-coefficients), whose cohomology is canonically isomorphic to $H^*(M;\Lambda)$.

{\em First approach (direct isomorphism).}
There is a classical argument \cite{floer-hofer-salamon94} which allows one to identify the Morse complex and Floer complex on the nose, assuming precise coordination of the choices involved in defining each of them. The main technical result from \cite{floer-hofer-salamon94} (with minor adaptations to our context) says that there is a function $H$ and compatible almost complex structure $J$ (not depending on any additional parameters), with the following properties:
\begin{itemize} \itemsep.5em
\item Along the boundary, $H = \epsilon\scrH$ to second order, for some small $\epsilon>0$, and $J = \scrJ$ to first order.

\item $H$ is Morse, and its gradient flow with respect to the metric $\omega(\cdot, J\cdot)$ is Morse-Smale. 

\item All $1$-periodic orbits of the Hamiltonian vector field $X$ of $H$ are constant.

\item Any non-stationary solution of Floer's equation \eqref{eq:floer} (with the given $t$-independent choice of $H$ and $J$) has nonnegative expected dimension, meaning that $\mathrm{index}(D_u) > 0$. Moreover, the solutions with $\mathrm{index}(D_u) = 1$ are all $t$-independent, and regular (this means that they are negative gradient flow lines of $H$; since $H$ is small, they will also be regular in the Morse-theoretic sense).

\item All (non-constant) $J$-holomophic spheres avoid the critical points of $H$, as well as its isolated gradient flow lines.
\end{itemize}
Note that we are not claiming that all $u$ are regular (it might be possible to get such a stronger statement using more sophisticated techniques \cite{wendl16}, and that would simplify our argument a little; but it is not necessary). In spite of that, one can define $\mathit{HF}^*(M,\epsilon)$ using the given $(H,J)$, and it will be canonically isomorphic to the standard definition, by a continuation map argument. Obviously, for this special choice, we have an identification of chain complexes
\begin{equation} \label{eq:chain-iso}
\mathit{CM}^*(H) = \mathit{CF}^*(H),
\end{equation}
and hence $H^*(M;\Lambda) \iso \mathit{HF}^*(M,\epsilon)$. A weakness of this approach is that it is not a priori clear whether this isomorphism is canonical; but as we will see, one can work around this issue, at the price of imposing additional conditions on $(H,J)$.

{\em Second approach (PSS map).}
In the construction of the PSS map \cite{piunikhin-salamon-schwarz94}, the Morse and Floer sides are a priori unrelated. We work with some choice of $(f,g)$ to define Morse cohomology; and some $\epsilon$, which can be arbitrary except for \eqref{eq:epsilon}, and $(H_t,J_t)$ to define Floer cohomology.  Consider solutions of the following equation:
\begin{equation} \label{eq:pss-equation}
\left\{
\begin{aligned}
& u: (\bR \times S^1) \cup \{+\infty\} \longrightarrow M, \\
& z: [0,\infty) \longrightarrow M, \\
& u(+\infty) = z(0), \\
& \partial_s u + J_{s,t}^{B}(\partial_t u - X_{s,t}^{B}) = 0, \\
& \partial_s z + \nabla_{g_s^B} f_s^{B} = 0, \\
& \textstyle\lim_{s \rightarrow -\infty} u(s,\cdot) = x, \\
& \textstyle\lim_{s \rightarrow +\infty} z(s) = y.
\end{aligned}
\right.
\end{equation} 
Here, $(\bR \times S^1) \cup \{+\infty\}$ is a partially compactified cylinder, which is a Riemann surface isomorphic to the complex plane. The limit $x$ is a $1$-periodic orbit of $H$, while $y$ is a critical point of the Morse function $f$. The auxiliary data appearing in \eqref{eq:pss-equation} have the following form:
\begin{itemize} \itemsep.5em
\item 
$(H_{s,t}^{B}, J_{s,t}^{B}) = (H_t, J_t)$ for $s \ll 0$. For $s \gg 0$, the family $J_{s,t}^{B}$ extends smoothly over $+\infty$, and $H_{s,t}^{B}$ vanishes. The boundary behaviour of the almost complex structures is as usual. For the functions, we require that (to second order) $H_{s,t}^{B} = \chi(s)\scrH$, where $\chi(s)$ is a nonincreasing function, equal to $\epsilon$ for $s \ll 0$ and to $0$ for $s \gg 0$.

\item The function $f_s^B$ equals $f$ for $s \gg 0$, and also agrees with $f$ near $\partial M$. Similarly, the metrics satisfy $g_s^B = g$ for $s \gg 0$.
\end{itemize}
By arguing as in Lemma \ref{th:levi} (with $S = (\bR \times S^1) \cup \{+\infty\}$ and  $\beta = \chi(s) \mathit{dt}$), one sees that any solution $u$ remains in the interior of $M$. Therefore, the point $z(0)$ lies in the interior, which implies that the same holds for all of $z$. By counting (for generic choices of all the auxiliary data) isolated solutions of \eqref{eq:pss-equation}, with the usual signs and powers of the Novikov variable $q$, one defines a chain map
\begin{equation} \label{eq:chain-b}
B: \mathit{CM}^*(f) \longrightarrow \mathit{CF}^*(H).
\end{equation}
A similar construction, with an added parameter, shows that \eqref{eq:chain-b} is independent of all choices up to chain homotopy. In the same sense, it is compatible with the continuation maps that relate different choices of $(f,g)$ and $(H,J)$. Hence, the induced cohomology level map is canonical. Obviously, in this generality, it is not an isomorphism. 

\begin{lemma} \label{th:fhs-pss}
For sufficiently small $\epsilon>0$, there is a choice of time-independent $(H,J)$ to which \eqref{eq:chain-iso} applies, and for which the composition of that isomorphism and the Morse-theoretic continuation map $\mathit{CM}^*(f) \rightarrow \mathit{CM}^*(H)$ recovers \eqref{eq:chain-b} up to chain homotopy.
\end{lemma}

\begin{proof}[Sketch of proof]
The argument is essentially a retread of \cite{floer-hofer-salamon94}, hence will only be outlined.
We consider time-independent $(H,J)$, and similarly choose $(H^B,J^B)$ in \eqref{eq:pss-equation} to be $t$-independent, while not imposing any constraints on the Morse theory side. One can achieve that:
\begin{itemize} \itemsep.5em
\item $(H,J)$ has all the conditions required for \eqref{eq:chain-iso};

\item Any solution of \eqref{eq:pss-equation} has nonnegative expected dimension. Moreover, the solutions with expected dimension zero are all $t$-independent, and regular.

\item For any $s \in \bR \cup \{\infty\}$, all $J_s^B$-holomorphic spheres avoid the points $u(s,t)$, where $u$ is a solution of \eqref{eq:pss-equation} with expected dimension zero.
\end{itemize}
If one adopts such a choice, isolated solutions of \eqref{eq:pss-equation} reduce to broken flow lines, of the form
\begin{equation} \label{eq:morse-continuation}
\left\{
\begin{aligned}
& u: \bR \cup \{+\infty\} \longrightarrow M, \\
& \partial_s u = 0 \quad \text{for $s \gg 0$,} \\
& z: [0,\infty) \longrightarrow M, \\
& u(+\infty) = z(0), \\
& \textstyle\lim_{s \rightarrow -\infty} u(s) = x, \\
& \textstyle\lim_{s \rightarrow +\infty} z(s) = y.
\end{aligned}
\right.
\end{equation}
We have omitted the ODE which $u$ and $z$ satisfy (both are $s$-dependent gradient equations). While \eqref{eq:morse-continuation} may not be the standard definition of a Morse-theoretic continuation map, it is chain homotopic to that map. 
\end{proof}

As a consequence of Lemma \ref{th:fhs-pss}, the PSS map in that particular instance is an isomorphism on cohomology; on the other hand, it follows that then, \eqref{eq:chain-iso} agrees with the PSS map on cohomology, hence fits into the general framework of canonical isomorphisms.

\section{The $q$-connection\label{sec:q}}

This section is the core of the paper. We introduce operations on $S^1$-equivariant Hamiltonian Floer cohomology, which constitute a rudimentary Cartan homotopy formalism. Just like in the classical definition of the Gauss-Manin connection, or in Getzler's noncommutative geometry version, the $q$-connection arises by combining that formalism with ``naive'' differentiation.

\subsection{Structure of the equivariant theory\label{subsec:formal}}
We continue in the geometric setup of the previous section. Define 
\begin{equation} \label{eq:cf-eq}
\mathit{CF}^*_{\mathit{eq}}(H) = \mathit{CF}^*(H)[[u]], 
\end{equation}
where $u$ is a formal variable of degree $2$. We will introduce $\Lambda[[u]]$-linear endomorphisms of this space, of the form
\begin{align}
\label{eq:d-eq}
& d_{\mathit{eq}} = d + u\Delta + O(u^2) : \mathit{CF}^*_{\mathit{eq}}(H) \longrightarrow \mathit{CF}^{*+1}_{\mathit{eq}}(H), \\
\label{eq:lambda-eq}
& \lambda_{\mathit{eq}} = \lambda + O(u) : \mathit{CF}^*_{\mathit{eq}}(H) \longrightarrow \mathit{CF}^{*+1}_{\mathit{eq}}(H), \\
\label{eq:iota-eq}
& \iota_{\mathit{eq}} = \iota + O(u): \mathit{CF}^*_{\mathit{eq}}(H) \longrightarrow \mathit{CF}^{*+2}_{\mathit{eq}}(H).
\end{align}
They will satisfy a kind of Cartan homotopy formalism:
\begin{align}
\label{eq:deq2}
& d_{\mathit{eq}}^2 = 0, \\
\label{eq:lambdadeq}
& d_{\mathit{eq}} \lambda_{\mathit{eq}} + \lambda_{\mathit{eq}} d_{\mathit{eq}} = 0, \\
\label{eq:cartan}
& d_{\mathit{eq}} \iota_{\mathit{eq}} - \iota_{\mathit{eq}} d_{\mathit{eq}} = u\lambda_{\mathit{eq}}.
\end{align}
These equations are higher-order extensions of \eqref{eq:delta-squared}, \eqref{eq:delta-lambda}, and \eqref{eq:iota-lambda}, respectively; the higher-order correction terms include our original chain homotopies, which is why we get equalities. 

Before discussing their construction, let's note some consequences. Define $\mathit{HF}^*_{\mathit{eq}}(M,\epsilon)$ to be the cohomology of $d_{\mathit{eq}}$. This clearly sits in a long exact sequence \eqref{eq:u-sequence}. By \eqref{eq:lambdadeq}, $\lambda_{\mathit{eq}}$ induces an endomorphism of $\mathit{HF}^*_{\mathit{eq}}(M,\epsilon)$, and \eqref{eq:cartan} shows that this endomorphism vanishes after multiplication with $u$. In fact, \eqref{eq:cartan} implies that for an equivariant cocycle $x = x_0 + ux_1 + \cdots$,
\begin{equation}
\lambda_{\mathit{eq}}(x) = u^{-1} d_{\mathit{eq}} \iota_{\mathit{eq}}(x_0) + d_{\mathit{eq}} \iota_{\mathit{eq}}(x_1 + ux_2 + \cdots).
\end{equation}
This shows that the cohomology level map induced by $\lambda_{\mathit{eq}}$ is the composition of two maps from \eqref{eq:u-sequence} and the map induced by $\iota$, in the following order:
\begin{equation}
\xymatrix{ 
\ar@/^1.5pc/[rrr]^-{\lambda_{\mathit{eq}}} 
\mathit{HF}^*_{\mathit{eq}}(M,\epsilon) \ar[r] & \mathit{HF}^*(M,\epsilon) \ar[r]_-{\iota} & \mathit{HF}^{*+2}(M,\epsilon) \ar[r] & \mathit{HF}^{*+1}_{\mathit{eq}}(M,\epsilon).
}
\end{equation}
In a sense, this is disappointing, since it means that (on the cohomology level) $\lambda_{\mathit{eq}}$ is not a genuinely new operation in the equivariant theory, but rather derived from its relation with ordinary Floer cohomology.

\begin{remark} \label{th:all-classes}
For the arguments so far, we could have used any cohomology class on $M$ instead of $q^{-1}[\Omega]$. More systematically, one can generalize \eqref{eq:lambda-eq} and \eqref{eq:iota-eq} to operations
\begin{align} 
\label{eq:extended-lambda}
& C^* \otimes \mathit{CF}^*_{\mathit{eq}}(H) \longrightarrow \mathit{CF}^{*-1}_{\mathit{eq}}(H), \\
\label{eq:extended-iota}
& C^* \otimes \mathit{CF}^*_{\mathit{eq}}(H) \longrightarrow \mathit{CF}^{*}_{\mathit{eq}}(H),\end{align}
where $C^*$ is a suitable chain complex underlying $H^*(M;\Lambda)$ (to strictly generalize our approach, this complex should admit submanifolds as cycles; however, other choices, such as Morse homology, may be technically easier). The first of these is a chain map, and the second satisfies an analogue of \eqref{eq:cartan}. 

A short digression may be permitted at this point. Let's place ourselves in the context of Application \ref{th:contact}. Consider symplectic cohomology $\mathit{SH}^*(M)$, and its underlying chain complex $\mathit{SC}^*(M)$, as well as the equivariant versions $\mathit{SH}^*_{\mathit{eq}}(M)$ and $\mathit{SC}^*_{\mathit{eq}}(M)$. Then, one can construct operations
\begin{align} 
\label{eq:extended-lambda-2}
& [\cdot,\cdot]_{\mathit{eq}}: \mathit{SC}^*(M) \otimes \mathit{SC}^*_{\mathit{eq}}(M) \longrightarrow \mathit{SC}^{*-1}_{\mathit{eq}}(M), \\
\label{eq:extended-iota-2}
& \bullet_{\mathit{eq}}: \mathit{SC}^*(M) \otimes \mathit{SC}^*_{\mathit{eq}}(M) \longrightarrow \mathit{SC}^{*}_{\mathit{eq}}(M),
\end{align}
which satisfy
\begin{align}
& d_{\mathit{eq}}[x_1,x_2]_{\mathit{eq}} + [dx_1,x_2]_{\mathit{eq}} + (-1)^{|x_1|} [x_1,d_{\mathit{eq}}x_2]_{\mathit{eq}} = 0,  \\
& u[x_1,x_2]_{\mathit{eq}} - d_{\mathit{eq}}(x_1 \bullet_{\mathit{eq}}x_2) + dx_1 \bullet_{\mathit{eq}} x_2 + (-1)^{|x_1|} x_1 \bullet_{\mathit{eq}} d_{\mathit{eq}}x_2 = 0.
\end{align}
One recovers \eqref{eq:extended-lambda} and \eqref{eq:extended-iota} (in a suitable chain homotopy sense) from these by composing with PSS maps $C^* \rightarrow \mathit{SC}^*(M)$ in the first entry. On the cohomology level, the outcome is that $\mathit{SH}^{*+1}(M)$ acts on $\mathit{SH}^*_{\mathit{eq}}(M)$ via \eqref{eq:extended-lambda-2}; and that action becomes trivial after multiplying with $u$. This parallels the situation in noncommutative geometry, involving Hochschild cohomology acting on (negative) cyclic homology.
\end{remark}

\subsection{The equivariant differential}
Each construction in equivariant Floer cohomology amounts to setting up an infinite hierarchy of parametrized moduli spaces with suitable recursive properties. Our basic organizing principle will be to use the spaces of Morse trajectories from Section \ref{sec:morse} as parameter spaces. At least in the case of the equivariant differential, the construction is not new, but we reproduce it here since it serves as the model for all subsequent arguments. We will give two versions of the definition, where the difference is mainly one of the language used.

{\em First definition.} (This is close to the approach in \cite{bourgeois-oancea12}.)
At each point $w \in S^\infty$, choose a Hamiltonian and almost complex structure
\begin{equation} \label{eq:jeq}
(H^{\mathit{eq}}_w, J^{\mathit{eq}}_w),
\end{equation}
smoothly depending on $w$, and subject to the following conditions:
\begin{itemize}
\itemsep.5em
\item The choice should be invariant under shift: $(H_{\sigma(w)}^{\mathit{eq}},J_{\sigma(w)}^{\mathit{eq}}) = (H_w^{\mathit{eq}},J_w^{\mathit{eq}})$.

\item With respect to $q^{-1}(c_0) \iso S^1$, the restriction of $(H^{\mathit{eq}},J^{\mathit{eq}})$ to that fibre should agree with the previously chosen $(H_t,J_t)$.

\item In a neighbourhood of each $c_k$, the family $(H_w^{\mathit{eq}},J_w^{\mathit{eq}})$ should be preserved by parallel transport for our chosen connection $A$ (this makes sense because the connection is assumed to be flat locally near $c_k$).
\end{itemize}

Let $v$ be a negative gradient flow line, going from $c_k$ ($k>0$) to $c_0$. By parallel transport for the connection $A$, we get a trivialization of \eqref{eq:q-bundle} over $v$. This trivialization is a map $\tilde{v}$ fitting into a commutative diagram
\begin{equation} \label{eq:tilde-v}
\xymatrix{
\ar[d] \bR \times S^1 \ar[rr]^-{\tilde{v}} && S^\infty \ar[d]^-{q} \\
\bR \ar[rr]^-{v} && \bC P^\infty.
}
\end{equation}
Then,
\begin{equation} \label{eq:deq-data}
(H_{v,s,t}^{d_{\mathit{eq}}}, J_{v,s,t}^{d_{\mathit{eq}}}) = (H^{\mathit{eq}}_{\tilde{v}(s,t)}, J^{\mathit{eq}}_{\tilde{v}(s,t)})
\end{equation}
satisfies \eqref{eq:tau-hj} for $\tau = \alpha_k(v)$, where $\alpha_k$ are the functions defined in the first proof of Lemma \ref{th:alpha}. We consider the moduli space of pairs $(v,u)$, where $v$ is a flow line and $u$ a solution of the equation \eqref{eq:floer-2} associated to the data \eqref{eq:deq-data}, divided by common translation (which reparametrizes $v$ and $u$ simultaneously). Counting points in this space (with signs and powers of $q$, as in the definition of the Floer differential) yields an operation
\begin{equation} \label{eq:deq-k}
d_{\mathit{eq},k}: \mathit{CF}^*(H) \longrightarrow \mathit{CF}^{*+1-2k}(H).
\end{equation}
By the first part of Lemma \ref{th:alpha}, $d_{\mathit{eq},1}$ is indeed (a valid choice for) the BV operator. Because of the inductive structure of the boundary \eqref{eq:boundary-pk-d}, the maps \eqref{eq:deq-k} satisfy the equations
\begin{equation} \label{eq:pre-dsquared}
-d d_{\mathit{eq},k} - d_{\mathit{eq},k} d\,\, - \!\!\! \sum_{\substack{k_1+k_2 = k \\ k_1,k_2>0}} d_{\mathit{eq},k_1} d_{\mathit{eq},k_2} = 0.
\end{equation}
Concerning the signs in \eqref{eq:pre-dsquared}, the first two come directly from \eqref{eq:sign-1}, and the last one from \eqref{eq:added-koszul} together with Proposition \ref{th:orientations}(i). It is natural to extend the definition to $k = 0$ by setting $d_{\mathit{eq},0} = d$. Then, \eqref{eq:pre-dsquared} just says that the following expression satisfies \eqref{eq:deq2}:
\begin{equation} \label{eq:taylor}
d_{\mathit{eq}} = d_{\mathit{eq},0} + u d_{\mathit{eq},1} + \cdots
\end{equation}

{\em Second definition.}
The second version of the construction is a little more detached from the specific Morse-theoretic construction of the parameter spaces $P_k$. Instead of starting with $S^\infty$ and then pulling back data from there to the moduli spaces of flow lines, we use the fact that the $P_k$ carry tautological families \eqref{eq:tautological}, and make our choices directly on the total spaces of those families.
Fix arbitrary functions $\alpha_k$ which satisfy the properties from Lemma \ref{th:alpha}. Suppose that for each  $[v] \in P_k$ and preimage $y \in \scrP_k$, we have chosen a family 
\begin{equation} \label{eq:x-family}
(H_{y,t}^{d_\mathit{eq}}, J_{y,t}^{d_\mathit{eq}})_{t \in S^1}
\end{equation}
with the following properties:
\begin{itemize}
\itemsep.5em
\item
If $y$ is sufficiently close to some point $z \in \bar\scrP_k \setminus \scrP_k$, \eqref{eq:x-family} should agree with $(H_t,J_t)$ up to a rotation in $S^1$-direction, which means that $(H_{y,t}^{d_{\mathit{eq}}}, J_{y,t}^{d_{\mathit{eq}}}) = (H_{t+\tau(y)}, J_{t+\tau(y)})$. Moreover, the amount of rotation $\tau(y) \in S^1$ should be locally constant under the $\bR$-action on $\scrP_k$. In the case $z = y_-([v])$, this amount of rotation should be zero, while for $z = y_+([v])$ it should be equal to $\alpha_k([v])$.

\item
If we consider a point $([v_1],[v_2]) \in (P_{k_1} \setminus \partial P_{k_1}) \times (P_{k_2} \setminus \partial P_{k_2}) \subset \partial P_k$, and a preimage $y \in \scrP_k$ in the $\bR$-component belonging to $[v_1]$, then \eqref{eq:x-family} agrees with the corresponding family associated to $y$ as a point of $\scrP_{k_1}$. In the same situation, if $y$ lies in the $\bR$-component belonging to $[v_2]$, \eqref{eq:x-family} agrees with the corresponding family over $\scrP_{k_2}$ up to rotation by $\alpha_{k_1}([v_1])$ in $S^1$-direction. The analogous property holds for broken trajectories with more than two components.
\end{itemize}
These conditions can easily be met by a recursive construction. In a nutshell, because of the second condition, the restriction of \eqref{eq:x-family} to $\partial P_l$ is completely determined by the choices made for $P_k$, $k<l$; and one then extends that to the interior, using the fact that the space of overall choices is contractible (of course, additional care must be exercised near $\bar\scrP_l \setminus \scrP_l$). 

Fix $[v] \in P_k \setminus \partial P_k$, and choose a representative $v$. This is the same as fixing a parametrization of the preimage of $[v]$ in $\scrP_k$, which we write as $s \mapsto y(s)$. From \eqref{eq:x-family}, we then obtain an analogue of \eqref{eq:deq-data}, this time defined by
\begin{equation} \label{eq:new-h-v}
(H_{v,s,t}^{d_{\mathit{eq}}}, J_{v,s,t}^{d_{\mathit{eq}}}) = (H_{y(s),t}^{d_{\mathit{eq}}}, J_{y(s),t}^{d_{\mathit{eq}}}),
\end{equation}
which again satisfies \eqref{eq:tau-hj} with $\tau = \alpha_k([v])$. Suppose that we have a sequence $[v^i]$ of such flow lines, which converges to a broken flow line $([v_1],[v_2]) \in (P_{k_1} \setminus \partial P_{k_1}) \times (P_{k_2} \setminus \partial P_{k_2})$. Choose representatives $v^i$ and $v_1,v_2$ as well as $\sigma^i_1,\sigma^i_2 \in \bR$ with $\sigma^i_2 - \sigma^i_1 \rightarrow \infty$, such that
\begin{equation}
\left\{
\begin{aligned}
& v^i(s+ \sigma^i_1) \longrightarrow v_1(s), \\
& v^i(s + \sigma^i_2) \longrightarrow v_2(s)
\end{aligned}
\right.
\end{equation}
(in the sense of uniform $\smooth$-convergence on compact subsets). The corresponding convergence statement for \eqref{eq:new-h-v} says that (in the same sense as before)
\begin{equation} \label{eq:convergence-1}
\left\{
\begin{aligned}
& (H_{v^i,s+\sigma^i_1,t}^{d_{\mathit{eq}}}, J_{v^i,s+\sigma^i_1,t}^{d_{\mathit{eq}}}) \longrightarrow (H_{v_1,s,t}^{d_{\mathit{eq}}}, J_{v_1,s,t}^{d_{\mathit{eq}}}), \\
& (H_{v^i,s+\sigma^i_2,t}^{d_{\mathit{eq}}}, J_{v^i,s+\sigma^i_2,t}^{d_{\mathit{eq}}}) \longrightarrow (H_{v_2,s,t+\alpha_{k_1}([v_1])}^{d_{\mathit{eq}}}, J_{v_2,s,t+\alpha_{k_1}([v_1])}^{d_{\mathit{eq}}} ).
\end{aligned}
\right.
\end{equation}
If one fixes a sufficiently large $L$ and restricts attention to $i \gg 0$, then 
\begin{equation} \label{eq:convergence-2}
(H_{v^i,s,t}^{d_{\mathit{eq}}}, J_{v^i,s,t}^{d_{\mathit{eq}}}) = 
\begin{cases}
(H_t,J_t) & s \in (-\infty, \sigma^i_1 - L], \\
(H_{t+\tau^i}, J_{t+\tau^i}) & s \in [\sigma^i_1 + L, \sigma^i_2 - L], \\
(H_{t+\alpha_k([v^i])}, J_{t+\alpha_k([v^i])}) & s \in [\sigma^i_2 + L,\infty),
\end{cases}
\end{equation}
where the $\tau^i \in S^1$ themselves converge to $\alpha_{k_1}([v_1])$. Each case in \eqref{eq:convergence-2} corresponds to a region where $y^i(s)$ is close to a point of $\bar\scrP_k \setminus \scrP_k$; this point is $y_-([v^i])$ in the first case and $y_+([v^i])$ in the third case, explaining the particularly simple nature of the formulae given there. The combination of \eqref{eq:convergence-1} and \eqref{eq:convergence-2} describes the limiting behaviour of $(H_{v^i,s,t}^{d_{\mathit{eq}}}, J_{v^i,s,t}^{d_{\mathit{eq}}})$ uniformly on all of $\bR \times S^1$. This kind of description generalizes to limits which are broken flow lines with an arbitrary number of components.

To define $d_{\mathit{eq}}$, we again consider pairs $(v,u)$, where $[v] \in P_k \setminus \partial P_k$, and $u$ is a solution of the equation \eqref{eq:floer-2} associated to \eqref{eq:new-h-v}. Given our previous discussion of the limiting behaviour of \eqref{eq:new-h-v} as one approaches $\partial P_k$, it is clear how to implement the necessary compactness argument; a similar strategy applies to gluing issues.

We end by comparing the two versions of the definition: the first one is  a special case of the second. Namely, take $\alpha_k$ defined by parallel transport for the connection $A$. Supposing that \eqref{eq:jeq} have been chosen, we define \eqref{eq:x-family} by
\begin{equation}
(H_{y,t}^{d_\mathit{eq}}, J_{y,t}^{d_\mathit{eq}}) = (H^{\mathit{eq}}_{e^{2\pi i t}\tilde{v}(y)}, J^{\mathit{eq}}_{e^{2\pi i t}\tilde{v}(y)}),
\end{equation}
where $\tilde{v}(y) \in S^\infty$ is defined as follows: first use the evaluation map $\scrP_k \rightarrow \bC P^\infty$ to associate to $y$ a point $v(y) \in \bC P^\infty$, and then use parallel transport along the flow line from $c_k$ to $v(y)$ to determine a preferred lift $\tilde{v}(y) \in S^\infty$. Because the evaluation map is smooth, this indeed yields a smooth family, which satisfies all our desired properties; and if one then considers the associated data \eqref{eq:new-h-v}, those agree with \eqref{eq:deq-data}. 

\begin{remark}
One can also ask the converse question, namely whether the choice \eqref{eq:deq-data} is indeed less general than \eqref{eq:x-family} (making the second definition genuinely more flexible). The answer is yes, but a precise understanding of the amount of additional flexibility hinges on tricky technicalities. Roughly speaking, our first approach was to choose a (time-dependent) almost complex structure and Hamiltonian for each point of $\bC P^\infty$; and our second approach was to choose one such structure for each point on a nonconstant gradient flow line. Ignoring critical points (where we have extra constraints anyway), any point of $\bC P^\infty$ lies on a unique such flow line. However, the notions of smoothness used in the two versions are not the same: in the first one the smooth structure of $\bC P^\infty$ is used, while the second one involves the smooth structures on compactified trajectory spaces. 
\end{remark}

\subsection{The operation $\lambda_{\mathit{eq}}$}
The construction of this operation is entirely parallel to that of $d_{\mathit{eq}}$, but using $P_k^s$ instead of $P_k$. The additional information provided by having $P_k^s$ as a parameter space is used to implement an incidence condition. 

{\em First definition.}
Let's suppose that the first construction of the equivariant differential has been adopted, with data \eqref{eq:jeq}. Given a point $(v,w) \in P_k^s \setminus \partial P_k^s$, we use the same associated data \eqref{eq:deq-data} as before, but (for consistency) change notation to
\begin{equation} \label{eq:lambdaeq-data}
(H_{v,w,s,t}^{\lambda_{\mathit{eq}}}, J_{v,w,s,t}^{\lambda_{\mathit{eq}}}) = 
(H_{\tilde{v}(s,t)}^{\mathit{eq}}, J_{\tilde{v}(s,t)}^{\mathit{eq}}).
\end{equation}
One considers solutions of the associated equation \eqref{eq:floer-2}, with the condition
\begin{equation} \label{eq:beta-evaluation}
u(0,\beta_k^s(v,w)) \in q^{-1}\Omega.
\end{equation}
Here, $\beta_k^s$ are the maps from the first proof of Lemma \ref{th:alphabeta-s}. For $k = 0$ the flow line is constant, \eqref{eq:lambdaeq-data} reduces to $(H_t,J_t)$, and \eqref{eq:beta-evaluation} to \eqref{eq:incidence-2}, except for orientations: in \eqref{eq:incidence-2}, the point of evaluation on $\bR \times S^1$ goes in negative direction around the circle, whereas in \eqref{eq:beta-evaluation} it proceeds positively, assuming we have oriented $P_0^s$ as in Section \ref{subsec:topology}. To account for that discrepancy, we will define $\lambda_{\mathit{eq}}$ using {\em the opposite orientation of $P_k^s$ for all $k$.} In parallel to our previous discussion of the differential, this yields maps
\begin{equation} \label{eq:lambda-k}
\lambda_{\mathit{eq},k}: \mathit{CF}^*(H) \longrightarrow \mathit{CF}^{*+1-2k}(H),
\end{equation}
of which the simplest ($k = 0$) one agrees with $\lambda$. These maps satisfy
\begin{equation} \label{eq:pre-lambda}
\sum_{k_1 + k_2 = k} d_{\mathit{eq},k_1} \lambda_{\mathit{eq},k_2} +
\sum_{k_1 + k_2 = k} \lambda_{\mathit{eq},k_1} d_{\mathit{eq},k_2} = 0.
\end{equation}
The geometry underlying that relation is the description of the boundary \eqref{eq:boundary-pk-s}, with signs coming from our general conventions and Lemma \ref{th:orientations}(iii). One combines the $\lambda_{\mathit{eq},k}$ into a $u$-Taylor series to define $\lambda_{\mathit{eq}}$.

{\em Second definition.} Take functions $(\alpha_k^s,\beta_k^s)$ as in Lemma \ref{th:alphabeta-s}. Suppose that the equivariant differential has been defined using some choice \eqref{eq:x-family}. We then similarly proceed to choose, for each $(v,w) \in P_k^s$ and preimage $y \in \scrP_k^s$, data
\begin{equation}
\label{eq:x-family-2}
(H_{y,t}^{\lambda_\mathit{eq}}, J_{y,t}^{\lambda_\mathit{eq}})_{t \in S^1},
\end{equation}
subject to conditions that are entirely analogous to the previous ones. The only (fairly obvious) difference is that, on the preimage of a point of $\partial P_k^s$, the behaviour of \eqref{eq:x-family-2} is governed by the previous choices of the same kind (for lower values of $k$) on exactly one connected component (corresponding to the part of the broken flow line which contains the marked point), and by \eqref{eq:x-family} on the other components. For $(v,w) \in P_k^s \setminus \partial P_k^s$ one then defines, in analogy with \eqref{eq:new-h-v},
\begin{equation} \label{eq:lambdaeq-h}
(H_{v,w,s,t}^{\lambda_{\mathit{eq}}}, J_{v,w,s,t}^{\lambda_{\mathit{eq}}}) =
(H_{y(s),t}^{\lambda_\mathit{eq}}, J_{y(s),t}^{\lambda_\mathit{eq}}),
\end{equation}
where $y(s) \in \scrP_k^s$ is the preimage corresponding to the point $v(s)$ on the gradient flow line $v$. Equivalently but in slightly more abstract terms, the choice of $y(s)$ is normalized so that for $s = 0$, it gives back the canonical section $y_*: P_k^s \rightarrow \scrP_k^s$. The evaluation condition is again of the form \eqref{eq:beta-evaluation}, and the rest of the construction proceeds as before. In parallel with the situation of the differential, the first construction is a special case of the second one (this time, a lot more flexibility is allowed in the second approach; this was already true for the two proofs of Lemma \ref{th:alphabeta-s}).

\subsection{The operation $\iota_{\mathit{eq}}$}
This follows exactly the same method as for $\lambda_{\mathit{eq}}$, but using the spaces $R_k^b$ from \eqref{eq:break-3}. Because of the more abstract nature of those spaces, the second approach works better at this point, so we will stick to that. Namely, we choose functions as in Lemma \ref{th:alphabeta-p}, and consider the tautological family $\scrR_k^b$. At each point $y \in \scrR_k^b$, we choose
\begin{equation}
\label{eq:x-family-3}
(H_{y,t}^{\iota_\mathit{eq}}, J_{y,t}^{\iota_\mathit{eq}})_{t \in S^1},
\end{equation}
subject to the same conditions as before, and which of course must restrict to the corresponding family on $P_{k-1}^s$ over that boundary face. The definition of the associated maps \eqref{eq:floer-2} proceeds exactly as in \eqref{eq:lambdaeq-h}, and we use the same evaluation condition. This leads to operations
\begin{equation}
\iota_{\mathit{eq},k}: \mathit{CF}^*(H) \longrightarrow \mathit{CF}^{*+2-2k}(H),
\end{equation}
which reduce to $\iota$ for $k = 0$. Because of the boundary structure \eqref{eq:break-3}, these satisfy
\begin{equation}
\sum_{k_1+k_2=k} d_{\mathit{eq},k_1} \iota_{\mathit{eq},k_2} - \sum_{k_1+k_2=k} \iota_{\mathit{eq},k_1} d_{\mathit{eq},k_2} - \lambda_{\mathit{eq},k-1} = 0.
\end{equation}
The Koszul signs \eqref{eq:added-koszul} disappear, because one of the two parameter spaces involved is always even-dimensional. What remains is a single $-1$ sign in front of the $\iota_{\mathit{eq},k_1} d_{\mathit{eq},k_2}$ term, for $k_2>0$, which comes from Lemma \ref{th:r-orientation}. Finally, the sign in front of $\lambda_{\mathit{eq},k-1}$ arises from the orientation-reversal convention we adopted when defining that operation. 

\subsection{The differentiation property}
At this point, we pick up the thread initiated in \eqref{eq:concrete-lambda}, which justifies the special role afforded to the class $q^{-1}[\Omega]$ in our setup. Namely, let $\partial_q$ be the operation of differentiation in $q$-direction, acting on $\mathit{CF}^*(H)$. This makes sense because, as a free $\Lambda$-module, $\mathit{CF}^*(H)$ carries a canonical basis (up to signs). Differentiation does not commute with the boundary operator: instead, we have
\begin{equation} \label{eq:diff-d}
\partial_q d - d \partial_q = \lambda.
\end{equation}
To be precise, this holds exactly provided that $\lambda$ has been defined by counting solutions of Floer's equation with the additional condition \eqref{eq:incidence-2}; in fact, the left hand side is precisely \eqref{eq:concrete-lambda}. The property \eqref{eq:diff-d} is an instance of a general idea which, in \cite{seidel16}, was called the ``differentiation axiom'' (in Gromov-Witten theory, a corresponding property is implied by the divisor axiom). There is also an equivariant refinement:
\begin{equation} \label{eq:diff-deq}
\partial_q d_{\mathit{eq}} - d_{\mathit{eq}} \partial_q = \lambda_{\mathit{eq}}.
\end{equation}
This assumes that the first version of the definition of $d_{\mathit{eq}}$ and $\lambda_{\mathit{eq}}$ has been used. Given that, the proof is exactly the same as for \eqref{eq:diff-d}. As an immediate consequence of \eqref{eq:diff-deq} and \eqref{eq:cartan}, the homomorphism
\begin{equation} \label{eq:define-gm}
\begin{aligned}
& \Gamma_q: \mathit{CF}^*_{\mathit{eq}}(H) \longrightarrow \mathit{CF}^{*+2}_{\mathit{eq}}(H), \\
& \Gamma_q(x) = u \partial_q x + \iota_{\mathit{eq}}(x)
\end{aligned}
\end{equation}
is a chain map. We define the $q$-connection to be the induced map on cohomology, using the same notation for it. This definition clearly satisfies the property from \eqref{eq:q-connection}, already on the chain level. The commutativity of \eqref{eq:q-zero} is also obvious, since $\Gamma_q(x) = \iota(x) + O(u)$. We can also immediately address the remaining properties mentioned in Section \ref{subsec:u}. Namely, if $\omega$ is exact, one can choose the cycle $\Omega$ to be empty, in which case both $\lambda_{\mathit{eq}}$ and $\iota_{\mathit{eq}}$ vanish, yielding \eqref{eq:exact}. Similarly, if $\omega \cdot H_2(M;\bZ) = m\bZ$ for some integer $m \geq 2$, one can choose $\Omega$ so that all its components have multiplicities in $m\bZ$. Then, $\lambda_{\mathit{eq}}$ and $\iota_{\mathit{eq}}$ vanish modulo $m$, leading to \eqref{eq:exact-mod-m}.

\begin{remark} \label{th:equi-connection}
Let's briefly explain the expected situation for symplectic cohomology, continuing the discussion from Remark \ref{th:all-classes} (this is just an outline; the details, which would require a combination of the techniques from here and \cite{seidel16}, remain to be carried out). One has a chain homotopy
\begin{equation} \label{eq:rhotopy}
\lambda_{\mathit{eq}} \htp [k,\cdot]_{\mathit{eq}},
\end{equation}
where $k \in \mathit{SC}^2(M)$ is a cocycle representing the image of $q^{-1}[\Omega]$ under the PSS map. Denoting the chain homotopy in \eqref{eq:rhotopy} by $\rho_{\mathit{eq}}$, one would then reformulate the definition of the $q$-connection on $\mathit{SH}^*_{\mathit{eq}}(M)$ in the following equivalent (up to chain homotopy) way:
\begin{equation} \label{eq:gamma-q-with-bullet} 
\Gamma_q(x) \htp u \partial_q x + u \rho_{\mathit{eq}}(x) + k \bullet_{\mathit{eq}} x.
\end{equation}
To tie that discussion to \cite[Section 3]{seidel16}, suppose that $k$ is in fact nullhomologous, say $k = d\theta$. One can then further rewrite \eqref{eq:gamma-q-with-bullet} as
\begin{equation} \label{eq:connection-times-u}
\Gamma_q(x) \htp u(\partial_q x + \rho_{\mathit{eq}}(x) - [\theta,x]_{\mathit{eq}}).
\end{equation}
At this point, one can divide the entire right hand side by $u$, which produces a (degree 0) endomorphism $\nabla_{\mathit{eq}}$ of $\mathit{SH}^*_{\mathit{eq}}$ satisfying 
\begin{equation}
\begin{aligned}
& \nabla_{\mathit{eq}}(fx) = f \nabla_{\mathit{eq}} x + (\partial_q f)x, \\
& \Gamma_q = u\nabla_{\mathit{eq}}.
\end{aligned}
\end{equation}
Furthermore, setting $u = 0$ in the definition of $\nabla_{\mathit{eq}}$ reproduces the connection $\nabla^{-1}$ from \cite{seidel16}. 
\end{remark}

\subsection{Well-definedness}
In our discussion of \eqref{eq:exact} and \eqref{eq:exact-mod-m}, we have implicitly made use of a property which requires justification, namely, independence of the $q$-connection from the choice of representative $\Omega$, within a fixed class $[\Omega] \in H^2(M;A)$. Generally speaking, the same issue arises with respect to all the other choices made in the construction. Luckily, the necessary well-definedness statements can all be proved in a uniform and fairly routine way; we will not give the details, but we will explain how the construction is set up.

Suppose that we are given two choices $(H_{\pm}, J_{\pm}, \Omega_{\pm})$ of data underlying the definition of Floer cohomology. On $\bR$, choose a Morse function $f = f(p)$ with exactly three critical points, namely local minima at $p = \pm 1$ and a maximum at $p = 0$. Instead of \eqref{eq:floer}, we now consider coupled equations involving a gradient flow line of $f$:
\begin{equation} \label{eq:floer-3}
\left\{
\begin{aligned}
& u: \bR \times S^1 \longrightarrow M, \\
& z: \bR \longrightarrow \bR, \\
& \partial_s u + \tilde{J}_{z(s),t}(\partial_t u - \tilde{X}_{z(s),t}) = 0, \\
& \partial_s z + f'(z) = 0,
\end{aligned}
\right.
\end{equation}
with the obvious asymptotics. The main point is that the almost complex structure $\tilde{J}$ and Hamiltonian $\tilde{H}$ depend on the additional parameter $p \in \bR$. We ask for $(\tilde{H}_{p,t}, \tilde{J}_{p,t})$ to agree with $(H_{-,t}, J_{-, t})$ if $p < \epsilon$ (for some small $\epsilon>0$), and with $(H_{+,t},J_{+,t})$ for $p > 1-\epsilon$. We also assume that a cycle $\tilde{\Omega} \subset \bR \times M$ is given, which equals $\bR \times \Omega_-$ on $\{p<\epsilon\}$ and $\bR \times \Omega_+$ on $\{p > 1-\epsilon\}$. We then consider a Floer cochain space whose generators are pairs consisting of a critical point of $f$ and a one-periodic orbit of the Hamiltonian associated to that critical point. The differential counts solutions of \eqref{eq:floer-3} with weights given by the intersection number of the cylinder 
\begin{equation} \label{eq:graph-cylinder}
\begin{aligned}
& \bR \times S^1 \longrightarrow \bR \times M, \\
& (s,t) \longmapsto (z(s),u(s,t))
\end{aligned}
\end{equation}
with $\tilde{\Omega}$. The resulting complex, denoted by $\mathit{CF}^*(\tilde{H})$, can be written as a mapping cone
\begin{equation} \label{eq:2-cone}
\mathit{CF}^*(\tilde{H}) = \left\{ \mathit{CF}^*(H_-) \xrightarrow{-\mathit{id}}
\mathit{CF}^{*-1}(H_-) \xleftarrow{\text{continuation map}}
\mathit{CF}^*(H_+) \right\}.
\end{equation}

\begin{example}
Let's consider the case where $(H_{\pm},J_{\pm}) = (H,J)$ are the same (and we choose $(\tilde{H},\tilde{J})$ trivially), but with different $\Omega_{\pm}$. The solutions of \eqref{eq:floer-3} relevant for the continuation map, as defined in \eqref{eq:2-cone}, consist of a trivial cylinder $u(s,t) = x(t)$ together with the unique (up to translation) $z(s)$ connecting $p = 0$ and $p = 1$. Hence, the resulting map is precisely \eqref{eq:rescale}. 
\end{example}

As an immediate consequence of \eqref{eq:2-cone}, we have a chain homotopy commutative diagram
\begin{equation} \label{eq:h-101}
\xymatrix{
\mathit{CF}^*(H_-) && \ar[ll]_-{\text{continuation map}} \mathit{CF}^*(H_+)
\\
& 
\ar[ul]^-{\text{projection }}_{\htp} \ar[ur]_-{\text{ projection}}^{\htp}
\mathit{CF}^*(\tilde{H})
}
\end{equation} 
in which both projections are quasi-isomorphisms. To see how this viewpoint is useful for studying the well-definedness of various additional structures on Floer cohomology, take for instance the operations \eqref{eq:iota} on $\mathit{CF}^*(H_{\pm})$, denoting them by $\iota_{\pm}$. One can define a similar operation $\tilde{\iota}$ on $\mathit{CF}^*(\tilde{H})$, using a suitable version of \eqref{eq:floer-3} with evaluation constraints in $q^{-1}\tilde{\Omega}$. By construction, this operation will be strictly compatible with the projections in \eqref{eq:h-101}, which proves that the continuation map relates $\iota_{\pm}$ up to chain homotopy (of course, there is nothing miraculous about this: the desired chain homotopy is encoded into the definition of $\tilde{\iota}$).

The same idea can be used to show that equivariant Floer cohomology is independent of the choices made in its construction. One defines a version of the equivariant theory that is coupled to Morse theory as in \eqref{eq:floer-3}, denoted by $\mathit{CF}^*_{\mathit{eq}}(\tilde{H})$, which comes with projections to $\mathit{CF}^*_{\mathit{eq}}(H_{\pm})$. A filtration argument (by powers of the equivariant parameter $u$) shows that these projections are quasi-isomorphisms. One defines the equivariant continuation map, up to chain homotopy, by filling in the analogue of \eqref{eq:h-101}:
\begin{equation} \label{eq:h-101-eq}
\xymatrix{
\mathit{CF}_{\mathit{eq}}^*(H_-) \;\; && \ar@{-->}[ll]_-{\text{equivariant continuation map}} \;\; \mathit{CF}^*_{\mathit{eq}}(H_+)
\\
& 
\ar[ul]^-{\text{projection }}_{\htp} \ar[ur]_-{\text{ projection}}^{\htp}
\mathit{CF}^*_{\mathit{eq}}(\tilde{H})
}
\end{equation}
Following the same strategy as before, one can show that $\mathit{CF}^*_{\mathit{eq}}(\tilde{H})$ carries all the same structures as $\mathit{CF}^*_{\mathit{eq}}(H_{\pm})$, including the $q$-connection, in a way which is compatible with the projections. This implies compatibility of the $q$-connection with the equivariant continuation map.

\begin{remark} \label{th:non-contractible}
We have allowed only nullhomogous $1$-periodic orbits, since that simplifies the exposition a little. Let's see what modifications are necessary in order to drop that restriction. The resulting Floer cohomology group will come with a splitting
\begin{equation} \label{eq:direct-summand}
\bigoplus_{a \in H_1(M)} \mathit{HF}^*(M,\epsilon)_a,
\end{equation}
with the previous definition contained in this as the $a = 0$ summand. The same decomposition will apply in the equivariant case, and all the structures we are considering, including the $q$-connection, are compatible with the splitting.

To obtain a $\bZ$-grading on \eqref{eq:direct-summand}, one needs to choose a trivialization of the anticanonical bundle; in fact, only the homotopy class of the trivialization is important. Changing that class by $\alpha \in H^1(M;\bZ)$ has the effect of shifting the grading of each summand in \eqref{eq:direct-summand} by an even amount $2 \int_a \alpha$; see e.g.\ \cite{seidel99}.

The other issue has to do with the maps \eqref{eq:rescale} which relate different choices of $\Omega$. Let's suppose that we define this map using intersection numbers with an interpolating $\tilde{\Omega}$. Changing $\tilde{\Omega}$ by $\alpha \in H^2_{\mathit{cpt}}(\bR \times M;A) = H^1(M;A)$ has the effect of multiplying the map \eqref{eq:rescale} with
\begin{equation} \label{eq:topological-rescaling}
x \longmapsto q^{\int_a \alpha}x
\end{equation}
on each summand \eqref{eq:direct-summand}. Hence, the continuation maps are no longer quite canonical. One can try to cure the ambiguity by adding more data, but that is irrelevant for our purpose: $\Gamma_q$ is compatible with those maps for any choice of $\tilde{\Omega}$. 

The last-mentioned observation may seem paradoxical, and deserves some further explanation. An equivalent statement is that $\Gamma_q$ remains invariant under the automorphism \eqref{eq:topological-rescaling} of \eqref{eq:direct-summand}. By its connection property,
\begin{equation} \label{eq:gamma-rescaled}
q^{-\int_a\! \alpha}\, \Gamma_q (q^{\int_a \alpha}\, x) = \Gamma_q(x) + uq^{-1} \big(\textstyle\int_a \!\alpha\big)x.
\end{equation}
Hence, what we are saying is that the action of $u\!\textstyle\int_a \alpha \in \Lambda[[u]]$ on $\mathit{HF}^*_{\mathit{eq}}(M,\epsilon)_a$ is trivial for any $a$. Indeed, one can prove directly that this is the case (one possible proof goes via the formalism mentioned in Remark \ref{th:all-classes}). As a noteworthy consequence, if $a$ is a primitive non-torsion class, then $u$ acts trivially on $\mathit{HF}^*_{\mathit{eq}}(M,\epsilon)_a$, and hence the behaviour of $\Gamma_q$ on that summand is entirely determined by the cap action of $q^{-1}[\Omega]$ on $\mathit{HF}^*(M,\epsilon)_a$ (the forgetful map being injective).
\end{remark}

%
%

\subsection{Small Hamiltonians}
Suppose now that we have time-independent $(H,J)$, satisfying the properties which are necessary for \eqref{eq:chain-iso} to apply.
%
When defining equivariant Floer cohomology, one can choose all the data involved (either \eqref{eq:jeq} or \eqref{eq:x-family}, depending on the approach chosen) to be equal to $(H,J)$. This means that all the equations \eqref{eq:floer-2} which appear reduce to the standard Floer equation for $(H,J)$. If one looks at the parametrized moduli space which underlies $d_{\mathit{eq},k}$ for some $k > 0$, all its points have expected dimension
\begin{equation}
2k - 1 + \mathrm{index}(D_u) \geq 2k-1 > 0.
\end{equation}
As a consequence, if we use this setup to define the equivariant differential, then $d_{\mathit{eq},k} = 0 $ for all $k>0$, so that 
\begin{equation}
d_{\mathit{eq}} = d. 
\end{equation}
This means that \eqref{eq:cf-eq} is an isomorphism of chain complexes, and hence
\begin{equation} \label{eq:u-trivial}
\mathit{HF}^*_{\mathit{eq}}(M,\epsilon) = \mathit{HF}^*(M,\epsilon)[[u]].
\end{equation}

For $\lambda_{\mathit{eq}}$ and $\iota_{\mathit{eq}}$, one can use a similar approach. Namely, take some family $(H^{\mathit{fix}}_{s,t}, J^{\mathit{fix}}_{s,t})$, with $(s,t) \in \bR \times S^1$ as usual, and which agrees with our fixed $(H,J)$ for $|s| \gg 0$. What one can arrange is that all equations which appear in the definition of $\lambda_{\mathit{eq}}$ and $\iota_{\mathit{eq}}$ are of the form
\begin{equation} \label{eq:theta-floer}
\left\{
\begin{aligned}
& \partial_s u + J_{s,t+\theta}^{\mathit{fix}} (\partial_t u - X_{s,t+\theta}^{\mathit{fix}}) \partial_t u = 0, \\
& u(0,\theta) \in q^{-1}\Omega.
\end{aligned}
\right.
\end{equation}
In words, all the $(H^\ast,J^\ast)$ that we encounter in \eqref{eq:floer-2} are rotated versions of $(H^{\mathit{fix}}, J^{\mathit{fix}})$, and the amount of rotation is dictated by the position of the point at which the incidence condition is imposed; this is possible only because the limiting data $(H,J)$ are $t$-independent. As a consequence, parametrized moduli spaces again can't have any isolated points, which means that
\begin{align}
& \lambda_{\mathit{eq}} = 0, \\
& \iota_{\mathit{eq}} = \iota.
\end{align}
Hence, we have
\begin{equation}
\Gamma_q = u\partial_q + \iota.
\end{equation}
By construction, $\iota$ is the quantum cap product with $q^{-1}[\Omega]$. We have therefore shown that the isomorphism \eqref{eq:u-trivial} identifies the $q$-connection with $u\partial_q + q^{-1}[\Omega] \frown \cdot$. 

This establishes the last part of Theorem \ref{th:q}, but in a form that involves the a priori non-canonical isomorphism \eqref{eq:u-trivial}. Following the same idea in Section \ref{subsec:morse-floer}, we will now outline how to resolve that remaining issue. In general, one can (using spaces of half gradient flow lines as parameter spaces) define an equivariant version of the PSS map, which is a chain map
\begin{equation} \label{eq:beq-define}
B_{\mathit{eq}} = B + O(u): \mathit{CM}^*(f)[[u]] \longrightarrow \mathit{CF}^*_{\mathit{eq}}(H).
\end{equation}
The induced map $H^*(M;\Lambda[[u]]) \rightarrow \mathit{HF}^*_{\mathit{eq}}(M,\epsilon)$ is independent of all choices. It is an isomorphism whenever the ordinary PSS map is, thanks to an easy spectral sequence comparison argument. Now, one can find a time-independent $(H,J)$ for which all our previous argument goes through, such that the PSS map reduces to a Morse-theoretic continuation map, and all the higher terms in \eqref{eq:beq-define} vanish, for the same reason as in our discussion of \eqref{eq:u-trivial}. For that particular choice, it then follows that \eqref{eq:u-trivial} agrees with the cohomology level map induced by \eqref{eq:beq-define}, hence is after all part of the standard framework of canonical isomorphisms.

\begin{remark} \label{th:shift}
A natural next step would be to look at the following situation. Suppose that $M$ is a manifold with contact type boundary, such that the Reeb flow on $\partial M$ is $1$-periodic and extends to a Hamiltonian circle action on the whole of $M$. Let's assume that the circle action is ``Calabi-Yau'', which means that there is a trivialization of the anticanonical bundle which is $S^1$-invariant. In that case, a version of the isomorphisms from \cite{seidel96} (see more specifically \cite{ritter14}) yields $\mathit{HF}^*(M,1+\epsilon) \iso \mathit{HF}^*(M,\epsilon)$ for all $\epsilon$. In particular, if $\epsilon>0$ is small, one has
\begin{equation} \label{eq:1-plus-epsilon}
\mathit{HF}^*(M,1+\epsilon) \iso H^*(M;\Lambda).
\end{equation}
The $S^1$-equivariant version of this story appears to be more interesting, and closely related to the ``shift operators'' (studied e.g.\ in \cite{braverman-maulik-okounkov11}). The analogue of \eqref{eq:1-plus-epsilon} says that, still for small $\epsilon > 0$,
\begin{equation}
\mathit{HF}^*_{\mathit{eq}}(M,1+\epsilon) \iso H^*_{\mathit{eq}}(M;\Lambda),
\end{equation}
where the right hand side is equivariant cohomology for the circle action on $M$. Given that, it seems plausible to conjecture that the $q$-connection on $\mathit{HF}^*_{\mathit{eq}}(M, 1+\epsilon)$ should correspond to the equivariant quantum connection on the right hand side. That equivariant connection is interesting even in cases where ordinary Gromov-Witten invariants vanish (see e.g.\ \cite{maulik10, braverman-maulik-okounkov11}).
%
\end{remark}

\subsection{Finite analogues\label{subsec:ganatra}}
To round off our discussion of the $q$-connection, we would like to mention a conjectural analogue in which $S^1$ is replaced by a cyclic group $\bZ/p$ (this addresses a question raised by Ganatra). Cyclic symmetries exist for a much wider class of (not necessarily Hamiltonian) symplectic Floer cohomology groups. Our main point of reference is \cite{seidel14}, which only considers the case $p = 2$; hence, we will ultimately restrict to that case, even though this can be a bit misleading.
%

Let $\phi: M \rightarrow M$ be a symplectic automorphism, and $M_\phi$ its mapping torus. We require a strengthened form of \eqref{eq:cy}, which is that the fibrewise tangent bundle of $M_\phi \rightarrow S^1$ should have vanishing first Chern class. This is equivalent to saying that $\phi$ can be lifted to a graded symplectic automorphism \cite{seidel99}; we fix such a lift. Similarly, we assume that the cohomology class of the fibrewise symplectic form $\omega_\phi$ on $M_\phi$ is integral, and fix an integral lift $[\Omega_\phi]$. Finally, one has to make certain requirements on the behaviour near $\partial M$, which we omit (but see e.g.\ \cite{uljarevic14}). One can then define fixed point Floer cohomology $\mathit{HF}^*(\phi)$, as a $\bZ$-graded module over $\bZ((q))$. The Floer cohomology of iterates $\phi^p$ carries a canonical action of $\bZ/p$. That action can be refined to yield a $\bZ/p$-equivariant version of the theory. From now on, assume that $p$ is prime, and use $(\bZ/p)((q))$ rather than $\bZ((q))$ as coefficients for Floer cohomology. Let's denote the resulting version of the equivariant theory simply by $\mathit{HF}^*_{\mathit{eq}}(\phi^p)$, omitting any mention of the coefficients for the sake of brevity. It is a finitely generated module over $H^*(B \bZ/p; (\bZ/p)((q)))$. We denote by $u$ the standard degree $2$ generator of that ring (for $p = 2$, $u$ is the square of the degree $1$ generator, which we denote by $h$; this is of course no longer true for $p>2$, even though there is a relation between the two via Massey products). The conjectural analogue of the $q$-connection is an endomorphism satisfying the same condition as in \eqref{eq:q-connection},
\begin{equation} \label{eq:finite-q}
\Gamma_q: \mathit{HF}^*_{\mathit{eq}}(\phi^p) \longrightarrow \mathit{HF}^{*+2}_{\mathit{eq}}(\phi^p).
\end{equation}

We will not attempt to construct \eqref{eq:finite-q} here, but we can outline a bit of the formal skeleton of the construction. For simplicity, suppose from now on that $p  = 2$. Let $\mathit{CF}^*(\phi^2)$ be the chain complex underlying $\mathit{HF}^*(\phi^2)$. Using an incidence condition with $q^{-1}\Omega_\phi$, one defines
\begin{equation} \label{eq:lambda-iota-2}
\begin{aligned}
& \iota: \mathit{CF}^*(\phi^2) \longrightarrow \mathit{CF}^{*+2}(\phi^2), && d\iota - \iota d = 0, 
\\
& \lambda: \mathit{CF}^*(\phi^2) \longrightarrow \mathit{CF}^{*+1}(\phi^2), && d\lambda + \lambda d = 0, 
\end{aligned}
\end{equation}
much as before. The operations \eqref{eq:lambda-iota-2} don't depend on having the square of a map, but the next steps do: one has
\begin{equation} \label{eq:s-sigma}
\begin{aligned}
& \sigma: \mathit{CF}^*(\phi^2) \longrightarrow \mathit{CF}^*(\phi^2), && d\sigma + \sigma d = 0, \\
& \Sigma: \mathit{CF}^*(\phi^2) \longrightarrow \mathit{CF}^{*-1}(\phi^2), && d\Sigma + \Sigma d = \sigma^2 + \mathit{id}.
\end{aligned}
\end{equation}
The first of these maps induces an endomorphism on cohomology, and the second shows that this endomorphism is an involution (recall that we are in characteristic $2$, so signs don't matter). We can introduce further operations, which can be seen as measuring the failure of $\iota$ to be compatible with the $\bZ/2$-action:
\begin{equation} \label{eq:xi-map}
\begin{aligned}
& \xi: \mathit{CF}^*(\phi^2) \longrightarrow \mathit{CF}^{*+1}(\phi^2), 
&& d\xi + \xi d = \sigma \iota + \iota \sigma, \\
& \Xi: \mathit{CF}^*(\phi^2) \longrightarrow \mathit{CF}^*(\phi^2), 
&& d\Xi + \Xi d = \sigma \xi + \xi \sigma + \Sigma \iota + \iota \Sigma + \lambda.
 \end{aligned}
\end{equation}
The equivariant differential on $\mathit{CF}^*_{\mathit{eq}}(\phi^2) = \mathit{CF}^*(\phi^2)[[h]]$ is of the form
\begin{equation} \label{eq:deq-finite}
\begin{aligned}
& d_{\mathit{eq}} = d + h (id + \sigma) + h^2 \Sigma + O(h^3):
\mathit{CF}^*_{\mathit{eq}}(\phi^2) \longrightarrow \mathit{CF}^{*+1}_{\mathit{eq}}(\phi^2), 
&& d_{\mathit{eq}}^2 = 0.
\end{aligned}
\end{equation}
One would then define the discrete analogue of the $q$-connection on $\mathit{CF}^*_{\mathit{eq}}(\phi^2)$ by a formula
\begin{equation}
\Gamma_q = \iota + h\xi + h^2 (\partial_q + \Xi) + O(h^3):
\mathit{CF}^*_{\mathit{eq}}(\phi^2) \longrightarrow \mathit{CF}^{*+2}_{\mathit{eq}}(\phi^2).
\end{equation}

Let's conclude this sketch by mentioning why one might be interested in studying such operations. On the algebraic side, the theory of differential operators in finite characteristic is much richer than its characteristic $0$ counterpart (see e.g.\ \cite{katz72} for applications to Gauss-Manin connections). On the geometric side, one has the equivariant squaring map \cite{seidel14c}
\begin{equation} \label{eq:square}
\mathit{HF}^*(\phi) \stackrel{Q}{\longrightarrow} \mathit{HF}^{2*}_{\mathit{eq}}(\phi^2).
\end{equation}
This satisfies $Q(qx) = q^2Q(x)$, hence its image is a subspace over $(\bZ/2)((q^2))$. The kernel of $\Gamma_q$ is a subspace of the same kind. One can speculate that the composition of \eqref{eq:square} and $\Gamma_q$ should be zero, and then further consider the relation of such statements with localisation theorems as in \cite{seidel14c}.

%

\section{The $u$-connection\label{sec:u}}
This section adapts the previous arguments to prove the results stated in Section \ref{subsec:u}. To avoid repetition, much of the discussion will be presented in abbreviated form. The main difference can be expressed as follows. Originally, we worked in a situation where Floer cohomology groups were $\bZ$-graded, which was useful in simplifying technical aspects of pseudo-holomorphic curve theory, but played no fundamental role in our argument. This time, grading issues will be key to our discussion.

\subsection{Floer cohomology revisited}
We will again work with a manifold $M$ and $(\scrH,\scrJ)$ satisfying \eqref{eq:h} and \eqref{eq:convexity}, but now assume \eqref{eq:monotone}. Choose a codimension two cycle 
$C = m_1C_1 + \cdots + m_j C_j$, with $m_j \in \bZ$, which represents the first Chern class. On the complement of $C$, we fix a trivialization of the anticanonical bundle $K_M^{-1}$, such that the following holds. Suppose that $S$ is a compact oriented surface with boundary, $u: S \rightarrow M$ a map such that $u(\partial S) \cap C = \emptyset$, together with a section $\xi$ of $u^*K_M^{-1}$ which is nonzero on the boundary. Then,
\begin{equation} \label{eq:c1-rel}
\sum_{\xi^{-1}(0)} \pm 1 = u \cdot C + w(\xi|\partial S).
\end{equation}
Here, the left hand side is the count of zeros with the usual signs, and $w(\xi|\partial S)$ is the winding number (or degree) of $\xi|\partial S$ as a map $\partial S \rightarrow S^1$, defined using the given trivialization. (In an algebro-geometric context, one would get to \eqref{eq:c1-rel} by taking a rational section of $K_M^{-1}$ whose zeros and poles equal the divisor $C$, and using that for the trivialization.)

When defining the Floer cochain complex, we choose $(H_t,J_t)$ so that all $1$-periodic orbits of $H$ are disjoint from $C$. Then, each such orbit still has an index $i(x) \in \bZ$, but \eqref{eq:index} should now be replaced by
\begin{equation} \label{eq:index-2}
\mathrm{index}(D_u) = i(x_-) - i(x_+) + 2(u \cdot C).
\end{equation}
Because of \eqref{eq:monotone}, one also has a one-form $\theta$ on $M \setminus C$, such that for the resulting actions $A_H(x)$, the analogue of \eqref{eq:energy} holds:
\begin{equation} \label{eq:energy-2}
E(u) = \int_{\bR \times S^1} \|\partial_s u\|^2 = A_H(x_-) - A_H(x_+) + \gamma (u \cdot C)  + \int_{\bR \times S^1} (\partial_s H_{s,t}^*)(u(s,t)).
\end{equation}
Equivalently, one can define the normalized action as
\begin{equation}
\bar{A}_H(x) = A_H(x) - \frac{\gamma}{2} i(x),
\end{equation}
and then a combination of \eqref{eq:index-2} and \eqref{eq:energy-2} yields the familiar energy bound for solutions with a given index:
\begin{equation} \label{eq:new-energy}
E(u) = \bar{A}_H(x_-) - \bar{A}_H(x_+) + \frac{\gamma}{2} \mathrm{index}(D_u) + 
\int_{\bR \times S^1} (\partial_s H_{s,t}^*)(u(s,t)).
\end{equation}

Dropping Novikov coefficients, we define the Floer cochain complex as the $\bZ/2$-graded group (with the grading given by $i(x)$ mod $2$) 
\begin{equation} \label{eq:floer-z}
\mathit{CF}^*(H) = \bigoplus_x \bZ_x.
\end{equation}
Let's define operations $\lambda$ and $\iota$ in parallel with those in Section \ref{subsec:floer}, but using $C$ instead of $q^{-1}\Omega$ in all incidence conditions. In particular, the cohomology level map induced by $\iota$ is now the quantum cap product with $c_1(M)$. These operations, and the BV operator, still satisfy Lemmas \ref{th:iota-lambda}--\ref{th:delta-squared}. Consider the endomorphism
\begin{equation} \label{eq:pseudo-grading}
\begin{aligned}
& \mu: \mathit{CF}^*(H) \longrightarrow \mathit{CF}^*(H), \\
& \mu(x) = i(x) x.
\end{aligned}
\end{equation}
This is not compatible with the differential. Instead, assuming that $\lambda$ has been defined using the same $(H,J)$ as the Floer differential, one has an analogue of \eqref{eq:diff-d}, namely
\begin{equation} \label{eq:mu-d}
\mu d - d \mu= d - 2\lambda.
\end{equation}
To see why this is the case, let $u$ be a solution of \eqref{eq:floer} which contributes to the Floer differential. Then $\mathrm{index}(D_u) = 1$, and hence
\begin{equation} \label{eq:failure-of-grading}
i(x_-) - i(x_+) = 1 - 2(u \cdot C).
\end{equation}
With that in mind, we write
\begin{equation}
\mu d(x_+) - d \mu(x_+) = dx_+ - 2 \sum_u \pm (u \cdot C) x_-.
\end{equation}
The term $\pm (u \cdot C)$ counts (with signs) the possible ways of translating $u$ in $s$-direction, and introducing some $r \in S^1$, so that the incidence condition $u(0,-r) \in C$ is satisfied. Hence, that term is exactly the coefficient of $x_-$ in $\lambda(x_+)$.
 
\subsection{The equivariant theory}
We define $\mathit{CF}^*_{\mathit{eq}}(H) = \mathit{CF}^*(H)[[u]]$, as a $\bZ/2$-graded $\bZ[[u]]$-module. This carries the same formalism of operations as in Section \ref{subsec:formal}, again using $C$ instead of $q^{-1}\Omega$. Let's extend the operator \eqref{eq:pseudo-grading} $u$-linearly to $C^*_{\mathit{eq}}(H)$. The analogue of \eqref{eq:diff-deq}, which holds assuming appropriate choices in the definition of $\lambda_{\mathit{eq}}$, is
\begin{equation} \label{eq:mu-d-eq}
\mu \, d_{\mathit{eq}} - d_{\mathit{eq}} \mu = \sum_{k=0}^\infty \big(1 - 2k \big) u^k d_{\mathit{eq},k} - 2\lambda_{\mathit{eq}}.
\end{equation}
Solutions $u$ which contribute to $d_{\mathit{eq},k}$ have $\mathrm{index}(D_u) = -\mathrm{dim}(P_k) = 1-2k$, which explains the occurrence of that term; the rest is exactly as before. One can differentiate elements of $C^*_{\mathit{eq}}$ in $u$-direction in the obvious way, and this satisfies
\begin{equation}
\partial_u d_{\mathit{eq}} - d_{\mathit{eq}} \partial_u = \sum_k k u^{k-1} d_{\mathit{eq},k}.
\end{equation}
With that in mind, \eqref{eq:mu-d-eq} can be rewritten as
\begin{equation}
(2u\partial_u + \mu) d_{\mathit{eq}} - d_{\mathit{eq}} (2u\partial_u + \mu) = d_{\mathit{eq}} - 2\lambda_{\mathit{eq}}.
\end{equation}
It follows that the map
\begin{equation} \label{eq:define-gamma-u}
\begin{aligned}
& \Gamma_u: \mathit{CF}^*_{\mathit{eq}}(H) \longrightarrow \mathit{CF}^*_{\mathit{eq}}(H), \\
& \Gamma_u(x) =  2u^2\partial_u x + u\mu(x) - 2\iota_{\mathit{eq}}(x)
\end{aligned}
\end{equation} 
satisfies
\begin{equation}
\Gamma_u d_{\mathit{eq}} - d_{\mathit{eq}} \Gamma_u = u\, d_{\mathit{eq}}.
\end{equation}
%
Even though \eqref{eq:define-gamma-u} is not a chain map, it does induce a map on cohomology. We define the $u$-connection to be that induced map, which clearly satisfies the property from \eqref{eq:u-connection}. Commutativity of \eqref{eq:u-zero} is also obvious, because $\Gamma_u(x) = -2\iota(x) + O(u)$. The properties \eqref{eq:graded-gamma} and \eqref{eq:c1-mod-2m} follow in the same way as their counterparts for the $q$-connection. 

\begin{remark}
In this context, it is unproblematic to drop the assumption that the $1$-periodic orbits should be nullhomologous, leading to Floer cohomology groups as in \eqref{eq:direct-summand}. One wrinkle of the resulting discussion deserves some mention. Namely, suppose that $c_1(M) = 0$. Then, Floer cohomology admits a $\bZ$-grading, but that grading is not unique if one includes all $1$-periodic orbits, as already mentioned in Remark \ref{th:non-contractible}. In spite of that, \eqref{eq:graded-gamma} holds for any choice of $\bZ$-grading: given two choices, the difference between the resulting grading operators multiplies each summand $\mathit{HF}^*_{\mathit{eq}}(M,\epsilon)_a$ by $2\int_a \alpha$, for $\alpha \in H^1(M;\bZ)$; and that operation becomes trivial if multiplied by $u$. In particular, if $a$ is primitive and non-torsion, then $\Gamma_u = u\,\mathrm{deg} = 0$ on $\mathit{HF}^*_{\mathit{eq}}(M,\epsilon)_a$. Similar observations apply to \eqref{eq:c1-mod-2m}.
\end{remark}

Our next task is to explain \eqref{eq:ueq2}. Let's change the definition of $\mathit{CF}^*(H)$ to make it into a $\bZ$-graded module over $\bC[q,q^{-1}]$, where $|q| = 2$. In view of \eqref{eq:failure-of-grading}, this is done by defining the differential to be
\begin{equation} \label{eq:graded-d}
dx_+ = \sum_u \pm q^{u \cdot C} x_-.
\end{equation}
The same principle will be applied to all other operations. For instance, the graded version of the $u$-connection is an endomorphism of $\mathit{CF}^*_{\mathit{eq}}(H)$ of degree $2$, still given by \eqref{eq:define-gamma-u}. The $q$-connection can be defined as
\begin{equation}
\Gamma_q(x) = u \partial_q x + q^{-1} \iota_{\mathit{eq}}(x).
\end{equation}
If one assumes \eqref{eq:monotone-2}, then this is indeed the same as our original approach towards defining the $q$-connection (taking into account that we are using $C$, which represents $[\omega]$, instead than $q^{-1}\Omega$, which represented $q^{-1}[\omega]$). Clearly, one has
\begin{equation} \label{eq:gamma-gamma}
\Gamma_u + 2q \Gamma_q = u (\mu + 2u\partial_u + 2q \partial_q) = u \mathrm{deg},
\end{equation}
where $\mathrm{deg}$ is the grading operator, multiplying each element of $\mathit{CF}^*_{\mathit{eq}}(H)$ by its degree, exactly as in \eqref{eq:deg-mu}. The relation \eqref{eq:gamma-gamma} implies \eqref{eq:ueq2}.
What remains to be discussed is the polynomial version of equivariant Floer cohomology. For that purpose, we adapt the argument from \cite[Section 7]{seidel14c} to the $S^1$-equivariant case. Suppose as before that \eqref{eq:monotone-2} holds, meaning that $\gamma = 1$. Let's consider the definition of the equivariant differential. Maps $u$ that contribute to $d_{\mathit{eq},k}$ have $\mathrm{index}(D_u) = 1-2k$. The key point is to govern the last term in \eqref{eq:new-energy}, so that it grows less slowly than the index term, yielding an energy which becomes negative (implying that the relevant moduli spaces must be empty) for $k \gg 0$. For the equation \eqref{eq:floer} of the Floer differential itself, the problematic term vanishes; in the definition of the BV operator (at least, as we have approached it, which means avoiding Morse-Bott methods) it is necessarily nontrivial, but can be made arbitrarily small by choosing the Hamiltonians to be close to time-independent ones. More systematically, one has the following:

\begin{lemma} \label{th:epsilon}
Fix some constant $\delta>0$. Then, one can choose the data underlying the definition of the equivariant differential $d_{\mathit{eq}}$, such that the following holds. For any equation \eqref{eq:floer-2} which contributes to $d_{\mathit{eq},k}$ ($k > 0$),
\begin{equation} \label{eq:max}
\int_{\bR \times S^1} \mathrm{max}_{x \in M}(\partial_s H^*_{s,t}(x))  < \delta k.
\end{equation}
\end{lemma}

\begin{proof}
We follow the second construction of the equivariant differential. Suppose that, when defining Floer cohomology, one takes $(H_t)$ to be a small perturbation of a time-independent Hamiltonian. Then, one can certainly define $d_{\mathit{eq},1}$ (the BV operator) so that the $k = 1$ case of \eqref{eq:max} is satisfied. That prescribes what \eqref{eq:x-family} does over $\partial P_2$. If we extend \eqref{eq:max} to broken cylinders by adding up the relevant terms for each component, then the given choice over $\partial P_2$ satisfies that condition. Now, among all functions $(H^*_{s,t})$ satisfying \eqref{eq:tau-hj} for some $\tau$, those for which the left hand side of \eqref{eq:max} is less than a given constant form an open convex subset. Hence, when extending the choice of \eqref{eq:x-family} from $\partial P_2$ over the whole of $P_2$, one can arrange that \eqref{eq:max} remains true, by using partitions of unity. Openness is important since it allows us to achieve transversality while still satisfying the necessary bounds. The same inductive procedure is then repeated for higher $k$.
\end{proof}

For us, it is sufficient to take $\delta < 1$. Then, \eqref{eq:new-energy} shows that $d_{\mathit{eq}}$ is indeed polynomial in $u$, hence yields a differential on $\mathit{CF}^*_{\mathit{poly}}(H) = \mathit{CF}^*(H)[u]$. The same principle applies to $\lambda_{\mathit{eq}}$ and $\iota_{\mathit{eq}}$, hence to the definition of the $u$-connection.


\end{document}